\documentclass[11pt]{amsart}
\usepackage{amscd,amsmath,amssymb,amsfonts}
\usepackage{url,enumerate,color}
\usepackage{dsfont}
\usepackage[cmtip, all]{xy}

\usepackage[top=30mm,right=30mm,bottom=30mm,left=30mm]{geometry}

\usepackage{mathtools}
\usepackage{tikz-cd}
\usepackage{tikz}

\theoremstyle{plain}
\newtheorem{theorema}{Theorem}

\newtheorem{thm}{Theorem}[section]
\newtheorem{lem}[thm]{Lemma}
\newtheorem{corol}[theorema]{Corollary}

\newtheorem{conje}[theorema]{Conjecture}

\numberwithin{equation}{section}

\newcommand{\Tr}{{\rm Tr}}

\newcommand{\Stab}{{\rm Stab}}
\newcommand{\Ker}{{\rm Ker}}

\newcommand{\Aut}{\mathrm{Aut}}
\newcommand{\Out}{\mathrm{Out}}
\newcommand{\Irr}{\mathrm{Irr}}

\newcommand{\eps}{\epsilon}

\newcommand{\SL}{\mathrm{SL}}
\newcommand{\PSL}{\mathrm{PSL}}
\newcommand{\GL}{\mathrm{GL}}
\newcommand{\PGL}{\mathrm{PGL}}

\newcommand{\GU}{\mathrm{GU}}
\newcommand{\SU}{\mathrm{SU}}
\newcommand{\PSU}{\mathrm{PSU}}
\newcommand{\PGU}{\mathrm{PGU}}
\newcommand{\Sp}{\mathrm{Sp}}
\newcommand{\PSp}{\mathrm{PSp}}
\newcommand{\SO}{\mathrm{SO}}
\newcommand{\GO}{\mathrm{O}}

{

\newcommand{\F}{{\mathbb F}}

\newcommand{\Z}{{\mathbb Z}}

\newcommand{\ppd}{{\rm ppd}}
\newcommand{\ZB}{{\mathbf Z}}
\newcommand{\CB}{{\mathbf C}}
\newcommand{\NB}{{\mathbf N}}
\newcommand{\OB}{{\mathbf O}}

\newcommand{\QBS}{\mathsf{Q}}

\newcommand{\St}{\mathsf {St}}

\newcommand{\AAA}{\mathsf{A}}
\newcommand{\SSS}{\mathsf{S}}
\newcommand{\tw}[1]{{}^#1\!}

\renewcommand{\setminus}{\smallsetminus}
\newcommand{\meo}{\mathsf{meo}}

\begin{document}

\title{Non-congruence presentations of finite simple groups}

\author{William Y. Chen}
\address{William Y. Chen, Department of Mathematics\\
University of Illinois at Urbana-Champaign\\
Urbana, IL, 61801\\
U.S.A.}
\email{oxeimon@gmail.com}
\author{Alexander Lubotzky}
\address{Alexander Lubotzky, Faculty of Mathematics and Computer Science\\ 
The Weizmann Institute of Science\\
Rehovot 7610001\\ Israel}
\email{alex.lubotzky@mail.huji.ac.il}
\author{Pham Huu Tiep}
\address{Pham Huu Tiep, Department of Mathematics\\
    Rutgers University \\
    Piscataway, NJ 08854\\
    U.S.A.}
\email{pht19@math.rutgers.edu}

\thanks{The second author is grateful for the support of the ERC grant 882751, and by a research grant from the Center for New Scientists at the Weizmann Institute of Science.}
\thanks{The third author gratefully acknowledges the support of the NSF (grant DMS-2200850), the Simons Foundation, and the Joshua Barlaz Chair in Mathematics.}
\maketitle

\begin{abstract}
	We prove two results on some special generators of finite simple groups and use them to prove that every non-abelian finite simple group $S$ admits a non-congruence presentation (as conjectured in \cite{CLT}), and that if $S$ has a non-trivial Schur multiplier, then it admits a smooth cover (as conjectured in \cite{CFLZ}).
\end{abstract}

\tableofcontents

\section{Introduction}

In this paper we prove two theorems about some special generators of finite simple groups:

\begin{theorema}\label{main1}
Let $S$ be any finite non-abelian simple group. Then $S$ is generated by two elements $u$ and $v$ such that 
the integers $|u|$, $|v|$, and $|uv|$ are pairwise coprime.
\end{theorema}

\begin{theorema}\label{main2}
Let $S$ be any finite non-abelian simple group with nontrivial Schur multiplier.
Then $S$ admits a {\em smooth covering} in the sense of \cite{CFLZ}, i.e. there exist a finite quasisimple group $G$ with 
$S \cong G/\ZB(G)$, $\ZB(G)  \neq 1$, and two elements $x,y \in G$ such that $G = \langle x,y \rangle$, and each of the elements 
$x$, $y$, $xy$ has the same orders in $G$ and in $G/\ZB(G)$.  
\end{theorema}

These results are of some interest for their own sake, but our interest arose from their applications. Theorem 2 proves Conjecture 1.10 of \cite{CFLZ}, and Theorem 1 proves Conjecture 6.3 of \cite{CLT}. To explain this latter conjecture, we will need the notion of a \emph{non-congruence} presentation.

\newcommand{\Inn}{\operatorname{Inn}}
\newcommand{\Epi}{\operatorname{Epi}}
\newcommand{\ra}{\rightarrow}
\newcommand{\rightiso}{\stackrel{\sim}{\longrightarrow}}
\newcommand{\cH}{\mathcal{H}}
\newcommand{\spmatrix}[4]{\left[\begin{smallmatrix}#1&#2\\#3&#4\end{smallmatrix}\right]}
\newcommand{\pr}{\operatorname{pr}}

We begin with some generalities. Let $H$ be a group and $N\le H$ a finite index normal subgroup. Let $A := \Aut(H)$, and let
\begin{equation}
A_N := \{\gamma\in A\;|\;\gamma(N) = N \text{ and $\gamma$ induces the identity on $H/N$}\}.
\end{equation}

The group $A_N$ is the (principal) $H$-congruence subgroup of $A$ associated to $N$. A subgroup of $A$ is $H$-congruence if it contains $A_N$ for some normal finite index subgroup $N\le H$. Similarly, we say that a subgroup of $\Out(H) := \Aut(H)/\Inn(H)$ is $H$-congruence if it contains the image of some $A_N$. We should mention $H$ in the notion of ``congruence'', since the same abstract group can be the outer automorphism group of different groups $H$.

The example most relevant to us is the case of $H = H_1 := \Z^2$ and $H = H_2 := F_2$, the free group of rank 2. In both cases, $\Out(H)\cong\GL_2(\Z)$, the isomorphism being given as follows: there is a natural map $\Aut(F_2)\ra\Aut(F_2/F_2') = \Aut(\Z^2)$, which by a classical result of Nielsen induces an isomorphism $\Out(F_2)\rightiso\Aut(\Z^2) = \GL_2(\Z)$. In this case, to be consistent with the number-theoretic literature, we will instead work inside the index 2 subgroup $\Aut^+(F_2)\le\Aut(F_2)$ consisting of automorphisms which act with determinant $+1$ on $\Z^2$, and its image $\Out^+(F_2)\cong\SL_2(\Z)$ inside $\Out(F_2)\cong\GL_2(\Z)$. Accordingly, for a finite index normal subgroup $N\le F_2$, we define the principal $F_2$-congruence subgroups $\Gamma_N\le\Out^+(F_2)$ by
\begin{equation}\label{eq_Gamma_N}
\Gamma_N := \{\gamma\in \Aut^+(F_2) \;|\;\gamma(N) = N \text{ and $\gamma$ induces an inner automorphism of $F_2/N$}\}/\Inn(F_2).
\end{equation}

As in the general case, we say that a subgroup of $\Out^+(F_2)$ is $F_2$-congruence if it contains $\Gamma_N$ for some $N$.

The $\Z^2$-congruence subgroups of $\Out^+(F_2)\cong\SL_2(\Z)$, defined analogously, are the classical congruence subgroups, which are classically defined as the subgroups containing, for some integer $n\ge 1$, the principal level $n$ congruence subgroup
$$\Gamma(n) := \Ker(\SL_2(\Z)\ra\SL_2(\Z/n)).$$
Henceforth, a \emph{congruence} subgroup of $\SL_2(\Z)$ will refer to a classical (i.e., $\Z^2$-) congruence subgroup.
While the classical ($\Z^2$-)congruence subgroups have attained a legendary status in number theory via its connections to elliptic curves and modular forms, the well-known failure of the \emph{congruence subgroup property} for $\SL_2(\Z)$ states that $\SL_2(\Z)$ also admits finite index subgroups which are not ($\Z^2$-)congruence. Such subgroups, typically called ``non-congruence'', are far less well understood. Despite its name, a breakthrough result of Asada \cite{Asa01} using algebraic geometry, later translated into group theory in \cite{BER11, BEL17}, asserts that \emph{every} finite index subgroup of $\SL_2(\Z)$ is $F_2$-congruence. In particular, every non-congruence subgroup is $F_2$-congruence.

Asada's result led the first author to give moduli interpretations to non-congruence modular curves \cite{Chen18} in terms of ``non-abelian level structures'', and suggests the following fundamental question:

\begin{equation}\label{eq_question}\parbox{14cm}{
Given an epimorphism $\varphi : F_2\ra G$ with $G$ finite, when is $\Gamma_\varphi := \Gamma_{\ker\varphi}$ a classical congruence subgroup of $\SL_2(\Z)$?
}
\end{equation}

\smallskip


In what follows, an epimorphism $\varphi : F_2\ra G$ will also be called a (2-generator) presentation of $G$. We say that $\varphi$ is a {\em congruence} (resp. {\em non-congruence}) presentation of $G$ if $\Gamma_\varphi = \Gamma_{\ker\varphi}$ is a classical ($\Z^2$-)congruence subgroup of $\SL_2(\Z)$ or not.

It is easy to see that for abelian $G$, every (2-generator) presentation of an abelian group is congruence. A more subtle result, independently discovered by the first author in a joint work with Deligne \cite{CD17}, and in a different language by Ben-Ezra \cite{BE16}, is that the same also holds for any metabelian $G$, i.e., all presentations of finite 2-generated metabelian groups are congruence.

Computations done in \cite{Chen18} show that amongst 2-generated groups $G$ of order $\le 255$, all presentations of solvable length 4 groups are non-congruence, and solvable length 3 groups only sometimes admit non-congruence presentations. Moreover, among non-abelian finite simple groups of order $\le |\text{J}_1| = 175560$, where $\text{J}_1$ is the Janko group, the overwhelming majority of presentations are non-congruence\footnote{Data for all 2-generated groups of order $\le 255$, as well as the finite simple groups of order $\le |\text{Sz}(8)| = 29120$ can be found in the appendix of \cite{Chen18}.}. These results suggest the philosophy that the existence of non-congruence presentations is a measure of ``non-abelianness''. In accordance with this philosophy, we will use Theorem \ref{main1} to show

\begin{corol}[c.f. {\cite[Conjecture 6.3]{CLT}}]\label{cor_non-congruence}
	Every non-abelian finite simple group admits a \emph{totally} non-congruence presentation.
\end{corol}

Here, a presentation $\varphi : F_2\ra G$ is \emph{totally noncongruence} if it is noncongruence and the smallest congruence subgroup containing $\Gamma_\varphi$ is $\SL_2(\Z)$ (see \S\ref{section_noncongruence}). Corollary \ref{cor_non-congruence} strengthens and generalizes the results of \cite[\S4.4]{Chen18}, which showed that the groups $A_n$ ($n\ge 5$) and $\PSL_2(p)$ ($p\ge 5$ prime) admit non-congruence presentations.

Note that a single group can admit both congruence and noncongruence presentations. This can even happen for a non-abelian finite simple group, the smallest examples being that of $G = \PSU_3(\F_4),\PSU_3(\F_5)$. In both cases, $G$ admits a single $\Aut(G)$-orbit of congruence presentations, every other presentation being noncongruence.\footnote{These groups have order 62400 and 126000 respectively, and are the only non-abelian simple groups of order $\le |\text{J}_1|$ to admit congruence presentations. The total number of presentations is approximately $|G|^2$, see \cite{KL90, LS95}.} In particular, these exceptional congruence presentations are \emph{characteristic}, in the sense that their kernels are characteristic subgroups of $F_2$. In our recent paper \cite{CLT}, we showed that these examples belong to an infinite family, containing characteristic presentations of the groups $\PSL_3(\F_q)$ and $\PSU_3(\F_q)$ for all $q\ge 7$, constructed using appropriate specializations of the Burau representation of the braid group $B_4$.

Despite the infinitude of such examples, we expect that they are rare. Indeed, a congruence presentation of a nonabelian simple group would correspond to a functorial method of constructing, from an isogeny of elliptic curves $E_1\ra E_2$, a non-abelian finite simple cover of $E_2$, only ramified above the origin \cite{Chen18}. Accordingly, we suggest the following

\begin{conje} Almost all presentations of a non-abelian finite simple group $S$ are non-congruence, in the following sense:
$$\lim_{|S|\ra\infty}\frac{\#\text{non-congruence presentations}}{\#\text{all presentations}} = 1$$
Here, a presentation is understood to be an epimorphism $F_2\ra S$.
\end{conje}

Note that since almost every pair of elements of $S$ generate $S$ \cite{KL90, LS95}, the denominator is approximately $|S|^2$.

\textbf{Organization of the paper.} In section \S\ref{section_topologies} we will phrase the discussion above in the language of profinite ``congruence'' topologies on $\SL_2(\Z)$, which leads to a number of interesting questions. In section \S\ref{section_noncongruence} we will show how Corollary \ref{cor_non-congruence} follows from Theorem \ref{main1}. Finally, in the remaining sections we will prove Theorems \ref{main1} and \ref{main2}.


\section{Congruence topologies and further research directions}\label{section_topologies}

\newcommand{\fF}{\mathfrak{F}}
\newcommand{\fS}{\mathfrak{S}}
\newcommand{\fA}{\mathfrak{A}}
The notions and results described in the introduction can be viewed from the more general perspective of profinite topologies on $\SL_2(\Z)$, suggesting further questions. Let $\fF$ be a formation of finite groups, i.e., a class of finite groups closed under homomorphic images and subdirect product \cite[\S2.1]{RZ}. Important examples are:

\begin{itemize}
\item $\fF_{all}$ -- the class of all finite groups,
\item $\fS_\ell$	-- the finite solvable groups of derived length $\le \ell$, so $\fS_1 = \fA$ is the class of abelian groups and $\fS_2$ the class of metabelian groups,
\item $\fF_p$ -- the class of finite $p$-groups, when $p$ is a prime.
\end{itemize}

The formation $\fF$ defines a topology $T(\fF)$ on the free group $F_2$, by taking the normal subgroups $N\le F_2$ for which $F_2/N\in\fF$ as basis of open neighborhoods of the identity. This topology is preserved by $\Aut(F_2)$ and defines a topology on $\Out^+(F_2)\cong\SL_2(\Z)$, called the $\fF$-congruence topology. Specifically, a basis to this topology is given by the groups $\Gamma_N$ given in \eqref{eq_Gamma_N}:
$$
\Gamma_N := \{\gamma\in \Aut^+(F_2) \;|\;\gamma(N) = N \text{ and $\gamma$ induces the identity on $F_2/N$}\}/\Inn(F_2).
$$
where now $N$ runs over the finite index normal subgroups $N\le F_2$ with $F_2/N\in\fF$. We will call this topology the $\fF$-congruence topology of $\SL_2(\Z)$. Thus the classical congruence topology is the $\fS_1 = \fA$-congruence topology, and Asada's theorem \cite{Asa01} asserts that the $\fF_{all}$-congruence topology is equal to the full profinite topology. Clearly, if $\fF_1\subset\fF_2$, then the $\fF_2$-congruence topology is stronger (or equal) to the $\fF_1$-congruence topology.

The results of Chen-Deligne \cite{CD17} and Ben-Ezra \cite{BE16} mentioned in the introduction mean that the $\fS_1$-congruence topology is equal to the $\fS_2$-congruence topology. However, the computational results mentioned in the introduction show that the $\fS_3$-congruence topology is strictly stronger.

Another way to define the $\fF$-congruence topology of $\SL_2(\Z)$ is the following. Let $F_2^\fF$ be the pro-$\fF$ completion of $F_2$, i.e. the completion of $F_2$ w.r.t. the $\fF$-topology:
$$F_2^\fF := \varprojlim_{N\in\fF} F_2/N.$$
There is a canonical map $F_2\ra F_2^\fF$, which is injective if and only if $F_2$ is residually-$\fF$, i.e., if and only if $\cap_{N\in\fF} N = 1$. Since $\Aut(F_2)$ preserves the $\fF$-topology, this induces a map $\Aut(F_2)\ra\Aut(F_2^\fF)$, and hence also a map $\SL_2(\Z)\cong\Out^+(F_2)\ra\Out(F_2^\fF)$. The $\fF$-congruence topology of $\SL_2(\Z)$ is also the topology induced on $\SL_2(\Z)$ from the profinite group $\Out(F_2^\fF)$.\footnote{The map $\SL_2(\Z)\ra\Out(F_2^\fF)$ is not always injective, in which case the $\fF$-congruence topology of $\SL_2(\Z)$ would not be Hausdorff, but it is so for all the interesting cases. In particular, it is true for all the examples we will discuss.}

This suggests a few interesting problems:
\begin{itemize}
\item[(i)] In the sequence of the $\fS_\ell$-congruence topologies, are there infinitely many jumps? We do not even know if the $\fS_4$-congruence topology is strictly stronger than the $\fS_3$-congruence topology or if they are equal.
\item[(ii)] Let $\fS = \cup_{\ell=1}^\infty \fS_\ell$. Is the $\fS$-congruence topology strictly weaker than the $\fF_{all}$-congruence topology (= the full profinite topology) of $\SL_2(\Z)$?
\item[(iii)] It is easy to see that the $\fF_p$-congruence topology of $\SL_2(\Z)$ is stronger (or equal) to the topology induced from $\SL_2(\Z)\hookrightarrow\SL_2(\Z_p)$. At the same time it is weaker (or equal) to the topology of $\SL_2(\Z)$ induced on it from the full pro-$p$ topology on the congruence subgroup
$$\Gamma(p) := \Ker(\SL_2(\Z)\ra\SL_2(\Z/p))$$
Is it equal to the latter one? A positive answer would be a ``local $p$'' version of Asada's theorem.
\end{itemize}

\section{Theorem \ref{main1} implies Corollary \ref{cor_non-congruence}}\label{section_noncongruence}
In this section we recall (and give a cleaner proof of) the non-congruence criterion of \cite[Theorem 4.4.10]{Chen18}), which shows that coprime generation in the sense of Theorem \ref{main1} implies that the associated presentation is totally non-congruence, thereby proving Corollary \ref{cor_non-congruence}.

For a finite index subgroup $\Gamma\le\SL_2(\Z)$, its \emph{congruence closure}, denoted $\Gamma^c$, is the smallest (classical) congruence subgroup containing $\Gamma$. The discussion in the previous section implies that the (classical) congruence topology on $\SL_2(\Z)$ is induced by the topology on $\SL_2(\widehat{\Z})$. Moreover, if $\overline{\Gamma}$ denotes the closure of $\Gamma$ inside $\SL_2(\widehat{\Z})$, then we have
$$\overline{\Gamma}\cap\SL_2(\Z) = \Gamma^c$$
It follows that $\Gamma$ is congruence if and only if $\Gamma = \Gamma^c$. In the antipodal case, where $\Gamma$ is non-congruence and $\Gamma^c = \SL_2(\Z)$, we will further say that $\Gamma$ is \emph{totally non-congruence}. Thus a finite index $\Gamma\le\SL_2(\Z)$ is totally non-congruence if it is a proper subgroup of $\SL_2(\Z)$ which is dense in $\SL_2(\widehat{\Z})$, c.f. \cite[\S4.4]{Chen18}.

Corollary \ref{cor_non-congruence} immediately follows from the following non-congruence criterion:
\begin{thm}[Coprime generation implies totally non-congruence]\label{thm_non-congruence_criterion_A} Let $F_2$ be a free group of rank 2, with generators $a,b$. Let $G$ be a nontrivial finite group generated by $x,y,z$ satisfying $xyz = 1$. Suppose the orders $|x|,|y|,|z|$ satisfy the following property:
\begin{itemize}
\item[$(*)$] The integers $|x|,|y|,|z|$ are pairwise coprime.
\end{itemize}
Then the $\SL_2(\Z)$-stabilizer $\Gamma_{\varphi_{x,y}} := \Gamma_{\ker\varphi}$ of the surjection $\varphi_{x,y} : F_2\ra G$ sending $a,b\mapsto x,y$ is \emph{totally non-congruence}.
\end{thm}

Before giving the proof, we make some remarks:
\begin{itemize}
\item The coprimality condition $(*)$ is equivalent to the statement that the gcd $\delta$ of the set $\{|x||y|,|x||z|,|y||z|\}$ is 1. If $G$ is abelian, then all three pairwise least common multiples of $|x|,|y|,|z|$ are equal to the exponent $e(G)$, and hence the gcd $\delta$ would be divisible by $e(G)$ in this case. Thus $\delta$ can be viewed as a measure of non-abelianness of $G$ relative to the generating triple $(x,y,z)$, and the theorem links this non-abelianness to the non-congruenceness of the stabilizer $\Gamma_{\varphi_{x,y}}$.

\item As will be evident from the proof, the theorem has some room for flexibility. If the pairwise gcd's in $(*)$ are not too large compared to the index of $\Gamma_{\varphi_{x,y}}$, then we can still show that $\Gamma_{\varphi_{x,y}}$ is non-congruence, though maybe not totally non-congruence.
\end{itemize}

\begin{proof} Define $\varphi_{z,x}$ and $\varphi_{y,z}$ analogously to $\varphi_{x,y}$. We note that $\varphi_{x,y},\varphi_{z,x},\varphi_{y,z}$ lie in the same $\SL_2(\Z)\cong\Out^+(F_2)$-orbit, and hence their stabilizers $\Gamma_{\varphi_{x,y}},\Gamma_{\varphi_{z,x}},\Gamma_{\varphi_{y,z}}$ are conjugate. Further observe that
\begin{equation}\label{eq_parabolic}
\spmatrix{1}{|x|}{0}{1},\spmatrix{1}{0}{|y|}{1}\in\Gamma_{\varphi_{x,y}},\quad\spmatrix{1}{|z|}{0}{1},\spmatrix{1}{0}{|x|}{1}\in\Gamma_{\varphi_{z,x}},\quad\text{and}\quad \spmatrix{1}{|y|}{0}{1},\spmatrix{1}{0}{|z|}{1}\in\Gamma_{\varphi_{y,z}}
\end{equation}

For ease of notation, let $\Gamma_1,\Gamma_2,\Gamma_3$ be these three subgroups. Since $G$ is nontrivial, condition $(*)$ implies that $x,y,z$ are not all conjugate to each other, and hence the $\Gamma_i$'s are proper subgroups of $\SL_2(\Z)$. It remains to show that each $\Gamma_i$ is dense in $\SL_2(\widehat{\Z})$. For this, fix an integer $n$ with prime factorization $n = \prod_j p_j^{r_j}$, and consider the direct product decomposition
\begin{equation}\label{eq_direct_sum_decomposition}
\SL_2(\Z/n)\cong \prod_{j}\SL_2(\Z/p_j^{r_j})	
\end{equation}
The coprimality assumption $(*)$ implies that for each $j$, at least two of $|x|,|y|,|z|$ are coprime to $p_j$, and hence at least one of $\Gamma_1,\Gamma_2,\Gamma_3$ projects onto $\SL_2(\Z/p_j^{r_j})$. Since the $\Gamma_i$'s are conjugate, each $\Gamma_i$ projects onto $\SL_2(\Z/p_j^{r_j})$ for every $j$, so the image of each $\Gamma_i$ in the product \eqref{eq_direct_sum_decomposition} is a subdirect product. Since the direct factors have no nontrivial common quotients, this subdirect product is the entirety of $\SL_2(\Z/n)$, and hence each $\Gamma_i$ surjects onto $\SL_2(\Z/n)$. Since $n$ was arbitrary, this shows that each $\Gamma_i$ is dense in $\SL_2(\widehat{\Z})$, as desired.
\end{proof}

\section{Coprime generation of finite simple groups}\label{section_coprime}


In this section we first outline our strategy to prove Theorems \ref{main1} and \ref{main2}.
In fact, for many of the groups $S$ we will show that the elements $u,v$ desired in Theorem \ref{main1} can be chosen so that $|u|$, $|v|$, $|w|$ are (pairwise) primes. 
Also, to prove Theorem \ref{main1} for a given $S$, it suffices to prove that the conclusion of Theorem \ref{main1} holds for
some quasisimple group $G$ with $G/\ZB(G) \cong S$. In most of the cases, we will show that $G = \langle x,y \rangle$ so that $u=x\ZB(G)$ and 
$v = y\ZB(G)$ satisfy the conclusion of Theorem \ref{main1}; moreover, the orders of $x$, $y$, and $xy$ are all coprime to $|\ZB(G)|$, and hence $x$ and $y$ 
satisfy the conclusion of Theorem \ref{main2}. 

\smallskip
We will employ a case-by-case approach to prove Theorems \ref{main1} and \ref{main2}, relying on the Classification of Finite Simple Groups \cite{GLS}.
We work with some quasisimple cover $G$ of $S$, with $\ZB(G) \neq 1$ if $\mathrm{Mult}(S) \neq 1$, and aim to find two elements $x,y \in G$ such that 
\begin{enumerate}[\rm(A)]
\item {\it $r=|x|$ and $s=|y|$ are coprime} (in fact primes if possible). This condition ensures that the elements $u:=x\ZB(G)$ and $v:=y\ZB(G)$ of 
$S=G/\ZB(G)$ have coprime orders. 
\item {\it $x^G \cdot y^G$ contains every non-central element of $G$.} This condition implies that $u^S \cdot v^S$ contains every nontrivial element of 
$S$. To ensure this condition, by Frobenius' character formula, it suffices to show that 
\begin{equation}\label{sum1}
  \Bigl|\sum_{1_G \neq \chi \in \Irr(G)}\frac{\chi(x)\chi(y)\chi(z)}{\chi(1)}\Bigr| < 1
\end{equation}  
for any $z \in G \smallsetminus \ZB(G)$.
Now, if $r$ and $s$ are primes, then by Burnside's $p^aq^b$-theorem, $S$ contains an element $w$ of prime order $t \neq r,s$, and without loss of generality 
we may assume that $w=uv$. In general, we aim to show that {\it we can find some prime divisor $t \nmid rs|\ZB(G)|$ of $|S|$}, and hence an element 
$z \in G$ of order $t$ for which we may again assume that $z=xy$. 
\item {\it No maximal subgroup of $S$ can have order divisible by $D:=\mathrm{lcm}(|u|,|v|,t)$.} By our construction of $u,v$ in (A) and (B), this implies that 
$S=\langle u,v \rangle$, as desired in Theorem \ref{main1}. 
\end{enumerate}

Recall \cite{Zs} that if $a,d \in \Z_{\geq 2}$ then $a^d-1$ admits a {\it primitive prime divisor $\ell$}, i.e. a prime $\ell$ that divides 
$a^d-1$ but not $\prod^{d-1}_{i=1}(a^i-1)$, unless $(a,d) = (2,6)$ or $d=2$ and $a+1$ is a $2$-power. Any such prime divisor $\ell$ satisfies 
$\ell \geq d+1$. Among all primitive prime divisors for the given pair $(a,d)$, we will write $\ell=\ppd(a,d)$ to denote the largest one among them.

\medskip
The main bulk of the work is
to prove Theorem \ref{main1} for $S=\AAA_n$ or $S=G/\ZB(G)$, $G$ a quasisimple group of Lie type defined over a field $\F_q$ in characteristic 
$p$, so $q=p^f$ with $f \geq 1$.  With $q$ fixed, we will denote by $\Phi_m=\Phi_m(q)$, 
the value of the $m^{\mathrm{th}}$ cyclotomic polynomial evaluated at $q$. 

Most of the times, we will choose $x, y \in G$ to be suitable regular semisimple elements, and use \eqref{sum1} if necessary to show that (B) holds.
To prove (C), in particular for classical groups, first we make use of \cite{BHR} to rule out all maximal subgroups of $S$ if $S$ has low rank. 
For the remaining classical groups, we will use results of \cite{GPPS} to narrow down the list of maximal subgroups of $S$ that may have order divisible by
$D$.  For exceptional groups of Lie type, we will use results of \cite{LM} and the following lemma.

\begin{lem}\label{hurwitz}
Suppose that the simple group $S$ is a Hurwitz group, i.e. 
$$S = \langle x,y \mid x^2=y^3=(xy)^7 = 1 \rangle.$$ 
Then Theorem \ref{main1} holds for $S$.
In particular, Theorem \ref{main1} holds for $\AAA_n$ with $n \geq 168$ and for the following simple groups of Lie type:
\begin{enumerate}[\rm(i)]
\item $\tw2 G_2(q)$, $q = 3^{2a+1} \geq 27$.
\item $G_2(q)$, $q \geq 5$.
\item $\tw3 D_4(q)$, where $q=p^f \neq 4$ and $p \neq 3$.
\end{enumerate}
\end{lem}

\begin{proof}
The first statement is obvious. The alternating case follows from \cite[Corollary]{Con}. Next, cases (i) and (ii) follow from the main result of \cite{M1}, and (iii) follows from \cite{M2}.
\end{proof}

There are further results in literature showing that certain finite  groups of Lie type are Hurwitz. However, these results usually 
assume the simple, say classical, group in question has very large rank, or specific fields of definition. Since we aim to prove Theorem \ref{main1} for {\bf all} simple classical groups as well as a variation of Theorem \ref{main1} in the case $S$ has nontrivial Schur multiplier, we will follow our main outline instead.
More precisely, we will follow the outline above to prove Theorems \ref{main1} and \ref{main2} for all but a finite number of exceptional cases. These remaining cases will be individually verified using \cite{GAP} via randomized search over generating pairs, or deduced as a corollary of existing results. Due to the complexity and variety of the groups involved, we will for the most part limit our exposition to general remarks and will avoid writing down explicit generators.

\section{Proof of Theorems \ref{main1} and \ref{main2}: Alternating groups}

\begin{lem}\label{2-trans}
Let $n \geq 5$ be odd and  $X < \SSS_n$ a transitive subgroup that contains an $(n-2)$-cycle $g$. Then $X$ is $3$-transitive.
\end{lem}

\begin{proof}
By assumption, $X$ is transitive on $\Omega:=\{1,2, \ldots,n\}$, and $g$ fixes, say $1$ and $2$.
 If $X$ is imprimitive, then $g$ fixes the imprimitivity block $\Delta_1$ of $\Omega$ that contains $1$. Note that $t:=|\Delta_1|$ is odd and $1 < t < n$. Since $2 \leq t-1 < n-2$ and $g$ acts on $\Delta_1 \setminus \{1\}$, $g$ fixes every point in $\Delta_1 \setminus \{1\}$. Thus $g$ has $t \geq 3$ fixed points, a contradiction. 
 
 Now assume that $\Stab_X(1)$ is not transitive on $\Omega \setminus \{1\}$. Then, since $g\in \Stab_X(1)$ acts transitively on $\Omega \setminus \{1,2\}$
 and fixes $2$, $\Stab_X(1)$ fixes $2$ and hence $\Stab_X(1) \leq \Stab_X(2)$. By order comparison, we have  $\Stab_X(1) = \Stab_X(2)$. Now 
 for any two points $i \neq j$ in $\Omega$, we say that $x \sim y$ if and only if $\Stab_X(i) = \Stab_X(j)$. Then $\sim$ is an equivalence relation on $\Omega$ which is $X$-equivariant and nontrivial. The equivalence classes of $\sim$ give rise to a nontrivial imprimitive system on $\Omega$ preserved by $X$, and thus
 $X$ is imprimitive, a contradiction.  
 
 Thus $X$ is $2$-transitive. But it contains the $(n-2)$-cycle $g$, so it is $3$-transitive.
\end{proof}

Suppose $S=\AAA_n$ with $2 \nmid n \geq 5$. Then we choose $u \in S$ to be an $n$-cycle and $v \in S$ to be an $(n-2)$-cycle. Say $v$ fixes 
$1$ and $2$ while acting on $\Omega = \{1,2,\ldots,n\}$. By \cite[Theorem 7]{HKL},
$u^{\SSS_n}v^{\SSS_n} = S \setminus \{1\}$. Hence, conjugating $u$ and $v$ in $\SSS_n$ suitably, we may assume that $uv$ is a double transposition when 
$n =5$, and an $(n-4)$-cycle when $n \geq 7$. 
We claim that $H:= \langle u,v \rangle$ equals $S$. Indeed, when $n=5$, $H \leq \AAA_5$ has order divisible by $30$, and so $H=\AAA_5$.
Suppose $n \geq 7$. Then $H \ni u$ is transitive on $\Omega:=\{1,2, \ldots,n\}$. By Lemma \ref{2-trans} applied to 
$H \ni v$, $H$ is doubly transitive. Now the pointwise stabilizer $\Stab_H(1,2)$ contains $v$, so it is transitive.
Since $H$ is doubly transitive, we may assume that some $H$-conjugate $t$ of $uv$ fixes $1$ and $2$. Thus $t \in \Stab_H(1,2)$ and $t$ acts 
as an $(n-4)$-cycle on $\Omega \setminus \{1,2\}$. It follows from Lemma \ref{2-trans} that $\Stab_H(1,2)$ is doubly transitive on  
$\Omega \setminus \{1,2\}$, and so $H$ is $4$-transitive on $\Omega$. A well-known consequence of the classification of finite simple groups is 
that the only proper $4$-transitive permutation subgroups of $\AAA_n$, cf. \cite[Theorem 5.3]{Cam}, are Mathieu groups in their natural actions.
Thus, if $H < \AAA_n$, then $(n,H) = (11,M_{11})$ or $(23,M_{23})$, and in both cases $n-4$ does not divide $|H|$, a contradiction. 

This proves Theorem \ref{main1} for $S=\AAA_n$ with $2 \nmid n \geq 5$. To prove Theorem \ref{main2} for $2\nmid n \geq 19$, 
we will work with $G=2\SSS_n$, and 
consider $x \in G$ an inverse image of odd order of an $n$-cycle and $y \in G$ an inverse image of odd order of an $(n-2)$-cycle. Taking $z \in G$ an
inverse image of odd order of an $(n-4)$-cycle, we claim that $z \in x^Gy^G$. We will use \eqref{sum1}, but dividing the 
sum into characters of $G$ which are non-faithful, respectively faithful. Note that for any $g \in G$ we have
$$\sum_{\chi \in \Irr(G/\ZB(G))}|\chi(g)|^2 = |\CB_{G/\ZB(G)}(g\ZB(G))|,~~\sum_{\chi \in \Irr(G) \setminus \Irr(G/\ZB(G))}|\chi(g)|^2 = 
    |\CB_G(g)|-|\CB_{G/\ZB(G)}(g\ZB(G))|.$$
By \cite[Corrollary 3.1.2]{LST11}, there are exactly $n$  characters $\chi \in \Irr(G/\ZB(G))$ which are nontrivial at $x$, namely the ones labeled by 
hook partitions, all taking values $\pm 1$ at $x$. Furthermore, among these characters, 
the two labeled by $(n-1,1) \vdash n$ and $(2,1^{n-2}) \vdash n$ vanish at $z$, and all
other ones have degree $1$ or $\geq (n-1)(n-2)/2$. It follows 
from the Cauchy--Schwarz inequality that
$$\Bigl| \sum_{\chi \in \Irr(G/\ZB(G)),~\chi(1) > 1} \frac{\chi(x)\chi(y)\overline{\chi(z)}}{\chi(1)}\Bigr| \leq \frac{2\sqrt{48(n-2)(n-4)}}{(n-1)(n-2)}.$$ 
For the faithful, i.e. spin, characters, we have $|\chi(x)| \leq \sqrt{n}$ and $|\chi(1)| \geq 2^{(n-1)/2}$, so again by Cauchy--Schwarz,
$$\Bigl| \sum_{\chi \in \Irr(G) \setminus \Irr(G/\ZB(G))} \frac{\chi(x)\chi(y)\overline{\chi(z)}}{\chi(1)}\Bigr| \leq \frac{\sqrt{48n(n-2)(n-4)}}{2^{(n-1)/2}}.$$ 
Taking $n \geq 19$ we have
$$\Bigl| \sum_{\chi \in \Irr(G),~\chi(1) > 1} \frac{\chi(x)\chi(y)\overline{\chi(z)}}{\chi(1)}\Bigr| < 2.$$
It follows from \eqref{sum1} that $z \in x^Gy^G$. Conjugating $x$ and $y$ suitably, we may assume $z = xy$. Since we have shown above that $S = \langle x\ZB(G),y\ZB(G) \rangle$, we are done.

\smallskip
Suppose $S=\AAA_n$ with $2 \mid n \geq 6$. Then we choose $u \in S$ to be an $(n-1)$-cycle and $v \in S$ to be an $(n-3)$-cycle. By \cite[Theorem 7]{HKL},
$u^{\SSS_n}v^{\SSS_n} = S \setminus \{1\}$. Hence, conjugating $u$ and $v$ in $\SSS_n$ suitably, we may assume that $uv$ is a disjoint product 
\begin{enumerate}[\rm(i)]
\item of a $5$-cycle and an $(n-5)$-cycle when $n \equiv 2,4 \pmod{5}$ (this ensures that $5$ is coprime to $(n-1)(n-3)(n-5)$, and so 
$|u|$, $|v|$, $|uv|$ are pairwise coprime), and 
\item of a $2$-cycle and an $(n-2)$-cycle when $n \equiv 0,1,3 \pmod{5}$ (clearly, $|u|$, $|v|$, $|uv|$ are pairwise coprime by this choice).
\end{enumerate}
We claim that $H:= \langle u,v \rangle$ equals $S$. Indeed, if $H \ni u$ is intransitive on $\Omega=\{1,2, \ldots,n\}$,
then $H$ has to fix the only fixed point, say $1$ of $u$. But $uv \in H$ has no fixed point, a contradiction. So $H$ is transitive and contains the $(n-1)$-cycle 
$u$, and so $\Stab_H(1)$ is transitive on $\Omega \setminus \{1\}$.  
We may assume that some $H$-conjugate $v'$ of $v$ fixes $1$. Thus $v' \in \Stab_H(1)$ and $v'$ acts 
as an $(n-3)$-cycle on $\Omega \setminus \{1\}$. It follows from Lemma \ref{2-trans} that $\Stab_H(1)$ is doubly transitive on  
$\Omega \setminus \{1\}$, and so $H$ is $3$-transitive on $\Omega$. Suppose $n = 6$. Then $|H|$ is divisible by $\mathrm{lcm}(3,4,5)=60$. If 
$H \neq \AAA_6$, then it follows from \cite{Atlas} that $H \cong \AAA_5$, but then $H$ contains no element of order $4$, a contradiction.
So we may assume that $n \geq 8$ and $v \in \Stab_H(1,2,3)$, and 
now $\Stab_H(1,2,3)$ is transitive on $\Omega \setminus \{1,2,3\}$, which means that $H$ is $4$-transitive. 
Again using \cite[Theorem 5.3]{Cam}, we see that either
$H = \AAA_n$, or $(n,H)=(12,M_{22})$, $(24,M_{24})$. But $n-5$ does not divide $|H|$ in the two latter cases, so $H = \AAA_n$.

This proves Theorem \ref{main1} for $S=\AAA_n$ with $2 \mid n \geq 6$. To prove Theorem \ref{main2} for $2 \mid n \geq 20$, 
we will again work with $G=2\SSS_n$, and 
consider $x \in G$ an inverse image of odd order of an $(n-1)$-cycle and $y \in G$ an inverse image of odd order of an $(n-3)$-cycle. Taking $z \in G$ an
inverse image of odd order of the disjoint product of a $5$-cycle and an $(n-5)$-cycle, we claim that $z \in x^Gy^G$. As in the odd-$n$ case, we will use \eqref{sum1}, but dividing the sum into characters of $G$ which are non-faithful, respectively faithful. By \cite[Corrollary 3.1.2]{LST11}, there are exactly $n-1$  characters $\chi \in \Irr(G/\ZB(G))$ which are nontrivial at $x$, all taking values $\pm 1$ at $x$ and having degree $1$ or $\geq n(n-3)/2$. It follows 
from the Cauchy--Schwarz inequality that
$$\Bigl| \sum_{\chi \in \Irr(G/\ZB(G)),~\chi(1) > 1} \frac{\chi(x)\chi(y)\overline{\chi(z)}}{\chi(1)}\Bigr| \leq \frac{2\sqrt{30(n-3)(n-5)}}{n(n-3)}.$$ 
For the faithful, i.e. spin, characters, we have $|\chi(x)| \leq \sqrt{n-1}$ and $|\chi(1)| \geq 2^{(n-2)/2}$, so again by Cauchy--Schwarz,
$$\Bigl| \sum_{\chi \in \Irr(G) \setminus \Irr(G/\ZB(G))} \frac{\chi(x)\chi(y)\overline{\chi(z)}}{\chi(1)}\Bigr| \leq \frac{\sqrt{30(n-1)(n-3)(n-5)}}{2^{(n-2)/2}}.$$ 
Taking $n \geq 20$ we have
$$\Bigl| \sum_{\chi \in \Irr(G),~\chi(1) > 1} \frac{\chi(x)\chi(y)\overline{\chi(z)}}{\chi(1)}\Bigr| < 2.$$
It follows from \eqref{sum1} that $z \in x^Gy^G$. Conjugating $x$ and $y$ suitably, we may assume $z = xy$. Since we have shown above that $S = \langle x\ZB(G),y\ZB(G) \rangle$, we are done.

Thus we have proved Theorem \ref{main1} for $\AAA_n$ with $n \geq 5$, and Theorem \ref{main2} for $\AAA_n$ with $n \geq 19$.

\smallskip
Finally, we check Theorem \ref{main2} for the cases $\AAA_n$ with $5\le n\le 18$. For all such $n$, we checked using \cite{GAP} that the perfect central double cover of $\AAA_n$ admits a generating pair $x,y$ such that $|x|,|y|,|xy|$ are all odd. For $n\ne 6$, these generating pairs are Nielsen equivalent\footnote{Recall that two generating pairs $(x,y),(x',y')$ are \emph{Nielsen equivalent} if $(x',y')$ can be produced from $(x,y)$ via a sequence of moves of the form $(x,y)\mapsto (x,xy)$, $(x,y)\mapsto (xy,y)$, and $(x,y)\mapsto (x,y^{-1})$.} to the pair given by the \cite{GAP} command
$$\texttt{GeneratorsOfGroup(DoubleCoverOfAlternatingGroup(n));}$$
where one should replace \texttt{n} with the appropriate degree. For $n = 6$, the claimed generating pair is Nielsen equivalent to the pair given by
$$\texttt{GeneratorsOfGroup(AtlasGroup("2.A6"));}$$

\section{Proof of Theorems \ref{main1}  and \ref{main2}: Exceptional groups of Lie type}

We begin with the case $G=S=\tw 2 B_2(q)$, $q=2^{2a+1} \geq 8$. Then we choose $|x|=q+\sqrt{2q}+1$ and $|y|=q-1$. Checking the character table 
of $G$ \cite{Bur}, one sees that the only nontrivial irreducible character of $G$ that is non-vanishing at both $x$ and $y$ is the Steinberg character 
$\St$, of degree $q^2$. Since $|\St(x)|=|\St(y)|=1$ but $|\St(z)| < |\St(1)$ for any $1 \neq z \in G$, \eqref{sum1} holds. We can now choose $|w|=q-\sqrt{2q}+1$ to
fulfill (B). Now $D=(q^2+1)(q-1)$, and so (C) holds by \cite[Table 8.16]{BHR}.

\smallskip
Next we complete the case $G=S=G_2(q)$ with $q>2$. Suppose $q=3$. Using \eqref{sum1} and \cite{Atlas} one can check 
that  $u^Sv^S = S \setminus \{1\}$ for $u \in S$ of order $7$ and $v \in S$ of order $8$. In particular, we may assume that $w=uv$ has order
$13$. None of the maximal subgroups of $S$ can contain elements for all of the orders $7$, $8$, and $13$ \cite{Atlas}, whence $S=\langle u,v \rangle$.
Suppose $q=4$. Using \eqref{sum1} and \cite{Atlas} one can check 
that  $u^Sv^S = S \setminus \{1\}$ for $u \in S$ of order $13$ and $v \in S$ of order $15$. In particular, we may assume that $w=uv$ has order
$7$. None of the maximal subgroups of $S$ can have order divisible by $7 \cdot 13 \cdot 15$ \cite{Atlas}, whence $S=\langle u,v \rangle$.
The case $q \geq 5$ is already done in Lemma \ref{hurwitz}(ii).

\smallskip
Suppose $G=S=\tw3 D_4(q)$. By Lemma \ref{hurwitz}(iii), it suffices to consider the case $p>2$ or $q=4$. Then we choose $x \in G$ an element of order 
$\Phi_{12}$ (belonging to class $s_{14}$ in the notation of \cite{DMi}), and $y \in G$ of order $\Phi_1\Phi_2\Phi_6$. Among the irreducible characters of 
$G$, only five are not of $\ell_1$-defect $0$, where $\ell_1 = \ppd(p,12f)$. So only these five characters can be non-vanishing at $x$, and two of them
are $1_G$ and the Steinberg character $\St$. Two of the three remaining ones have $\ell_2$-defect $0$, where $\ell_2=\ppd(p,6f)$, and hence vanish 
at $y$. The fifth one, $\tw3 D_4[-1]$ in the notation of \cite{Sp}, also vanishes at $y$. Arguing as above, we see that \eqref{sum1} holds for any
$1 \neq z \in G$, and hence $x^Gy^G \supseteq G \setminus \{1\}$. Thus $B$ holds, and hence we can choose $x$ and $y$ in such a way that
$|xy|=p$; in particular, $|\langle x,y \rangle|$ has order divisible by $p\Phi_1\Phi_2\Phi_6\Phi_{12}$. Inspecting the list of maximal subgroups of
$G$ \cite[Table 8.51]{BHR}, we see that $G=\langle x,y \rangle$.

\smallskip
Suppose $G = S= \tw2 F_4(q)$ with $q>2$. By \cite[Lemma 2.13]{GT}, $G$ admits elements $x$ of order $\ppd(p,12f)$ and $y$ of order 
$\ppd(p,6f)$ such that $x^Gy^G = G \setminus \{1\}$. In particular, we may assume that $|xy|= \ppd(p,4f)$. Checking the list of maximal subgroups of
$G$ \cite{M3}, we conclude that $G=\langle x,y \rangle$. Next suppose that $S=\tw2 F_4(2)'$. Using \eqref{sum1} and \cite{Atlas} one can check 
that  $u^Sv^S = S \setminus \{1\}$ for $u \in S$ of order $13$ and $v \in S$ of order $10$. In particular, we may assume that $w=uv$ has order
$3$. None of the maximal subgroups of $S$ can contain elements for all of the orders $13$, $10$, and $3$ \cite{Atlas}, whence $S=\langle u,v \rangle$.

\smallskip
Suppose $G=S=F_4(q)$ with $q>2$. Then the proof of \cite[Theorem 3.1]{LM} shows that $G = \langle x,y \rangle$ for some $x$ of order $\Phi_{12}$ and 
$y$ of order $2$, with $|xy|=3$. In the case $G=2 \cdot F_4(2)$, we can choose $x \in G$ of class $17a$ and $y \in G$ of class $13a$, and $z \in G$ of
class $15a$, in the notation of \cite{GAP}. Then we can check using \cite{GAP} that $z \in x^Gy^G$, and so, conjugating $x$ and $y$ suitably, we 
may assume $z=xy$. According to \cite{Atlas}, no maximal subgroup of $S=G/\ZB(G)$ can have order divisible by $13 \cdot 17$, whence
$G = \langle x,y \rangle$, and thus we have proved Theorems \ref{main1} and \ref{main2} for $S = F_4(2)$. 

\smallskip
Suppose $G=E_6(q)_{\mathrm {sc}}$, so that $S=G/\ZB(G)$ with $|\ZB(G)=\gcd(3,q-1)$. Then the proof of \cite[Theorem 4.1]{LM} shows that $G = \langle x,y \rangle$ for some $x$ of order $\Phi_{9}$ and $y$ of order $2$, with $|xy|=3$. Now if $3 \nmid (q-1)$, then $\gcd(6,\Phi_9)=1$, and so we are done.
Suppose $3|(q-1)$. Note that $T:=\langle x \rangle$ is a cyclic maximal torus of $G$, so $T > \ZB(G) \cong C_3$. It follows that $u=x\ZB(G)$ has order 
$(\Phi_9)/3$ which is coprime to $6$, and hence we have proved Theorem \ref{main1} in this case.

\smallskip
Suppose $G=\tw2 E_6(q)_{\mathrm {sc}}$, so that $S=G/\ZB(G)$ with $|\ZB(G)=\gcd(3,q+1)$. Then the proof of \cite[Theorem 5.1]{LM} shows that $G = \langle x,y \rangle$ for some $x$ of order $\Phi_{18}$ and $y$ of order $2$, with $|xy|=3$. Now if $3 \nmid (q+1)$, then $\gcd(6,\Phi_{18})=1$, and so we are done. 
Suppose $3|(q+1)$. Note that $T:=\langle x \rangle$ is a cyclic maximal torus of $G$, so $T > \ZB(G) \cong C_3$. It follows that $u=x\ZB(G)$ has order 
$(\Phi_{18})/3$ which is coprime to $6$, and hence we proved Theorem \ref{main1} in this case.

\smallskip
Using the recent result of Craven \cite{Cr1}, we can give an alternate argument, which also proves the conclusion of Theorem \ref{main2}, in the cases where $G=E_6(q)_{\mathrm {sc}}$ and $\tw2 E_6(q)_{\mathrm {sc}}$, 
so that $S=G/\ZB(G)$ with $|\ZB(G)=\gcd(3,q-1)$, respectively $|\ZB(G)|=\gcd(3,q+1)$. 
As mentioned in \cite[\S2.2.5]{GT}, $G$ contains regular semisimple elements $x$ of order $\ell_1=\ppd(p,8f)$, and $y$ of order
$\ell_2$, where $\ell_2=\ppd(p,9f)$ if $G$ is of type $E_6$ and $\ell_2 = \ppd(p,18f)$ otherwise, such that $x^Gy^G = G \setminus \ZB(G)$. Conjugating $x$ and $y$ suitably, we may assume that  $|xy|=\ell_3=\ppd(p,12f)$. Note that $\ell_1 \geq 17$ and $\ell_2 \geq 19$. Assume the contrary that $\langle u,v \rangle \leq M$ for some maximal subgroup $M$ of $S$. Using the list of (possible)
maximal subgroups of $S$ as given in \cite[Theorem 1.1]{Cr1} and the fact that $|M|$ is divisible by $\ell_1\ell_2\ell_3$, we see that 
$M = \NB_S(R)$, where $R$ is some non-abelian simple group with either $\ell_1 \nmid |\Aut(R)|$ or $\ell_2 \nmid |\Aut(R)|$, unless
$$(G,R,\ell_1,\ell_2)=(E_6(4)_{\mathrm {sc}},J_3,17,19).$$ 
Note that in fact the latter case cannot occur since we have $\ell_1=257$ when $q=4$.
In the former case, either 
$x$ or $y$ centralizes $R$. In other words, there is some element $a \in \{x,y\}$ such that $R \leq \CB_S(a)$, which is impossible since $\CB_G(a)$ is a torus.

\smallskip
Next we consider the case $G=E_7(q)_{\mathrm {sc}}$, so that $S=G/\ZB(G)$ with $|\ZB(G)=\gcd(2,q-1)$. 
As mentioned in \cite[\S2.2.5]{GT}, $G$ contains regular semisimple elements $x$ of order $\ell_1=\ppd(p,18f)$ and $y$ of order
$\ell_2=\ppd(p,7f)$ such that $x^Gy^G = G \setminus \{1\}$. Conjugating $x$ and $y$ suitably, we may assume that 
$|xy|=\ell_3=\ppd(p,14f)$. Note that $\ell_2 \geq 19$ and $\max(\ell_1,\ell_3) \geq 43$. Assume the contrary that $\langle u,v \rangle \leq M$ for some maximal subgroup $M$ of $S$. Using the list of (possible)
maximal subgroups of $S$ as given in \cite[Theorem 1.1]{Cr2} and the fact that $|M|$ is divisible by $\ell_1\ell_2\ell_3$, we see that 
$M = \NB_S(R)$, where $R = \PSL_2(t)$ with $t \in \{7,8,9,13\}$ or $R = \PGL_3(q)$, $\PGU_3(q)$. In particular, $\ell_1 \nmid |\Aut(R)|$, and so
$x$ centralizes $R$. In other words, $R \leq \CB_S(x)$, which is impossible since $\CB_G(x)$ is a torus.

\smallskip
Suppose $G=S=E_8(q)$. Then the proof of \cite[Proposition 8.1, Theorem 8.2]{LM} when $p=3$, and the proof of \cite[Theorem 9.1]{LM} when 
$p \neq 3$, show that $G = \langle x,y \rangle$ for some $x$ of order $\Phi_{30}$ and 
$y$ of order $2$, with $|xy|=3$. Note that $\gcd(6,\Phi_{30})=1$, so we are done.

\smallskip
Finally, we use \cite{GAP} to check Theorem \ref{main2} for the remaining cases $G_2(3), G_2(4)$, and $\tw2 B_2(8)$. The Schur multiplier of $G_2(3)$ has order 3, and the Schur cover admits a generating pair $x,y$ with $|x| = |y| = |xy| = 13$, and which is Nielsen equivalent to the generating pair given by
$$\texttt{GeneratorOfGroup(AtlasGroup("3.G2(3)"))};$$
The group $G_2(4)$ has a Schur multiplier of order 2, and the Schur cover admits a generating pair $x,y$ with $|x| = |y| = 13, |xy| = 5$, and which is Nielsen equivalent to the pair
$$\texttt{GeneratorOfGroup(AtlasGroup("2.G2(4)"))};$$
The group $\tw2 B_2(8)$ has Schur multiplier isomorphic to $C_2\times C_2$, and the Schur cover admits a generating pair $x,y$ with $|x| = |y| = 7, |xy| = 5$, and which is Nielsen equivalent to the pair
$$\texttt{GeneratorOfGroup(AtlasGroup("2\^{}2.Sz(8)"))};$$

\section{Proof of Theorems \ref{main1}  and \ref{main2}: Linear and unitary groups}

In this section, we prove Theorem \ref{main1} for $S = \PSL_n(q)$, $\PSU_n(q)$, with $n \geq 3$.

\subsection{Low rank groups}\label{slu-small}
We begin with the case $G = \SL_2(q)$, $q = p^f \geq 4$. Then we choose $x|=(q+1)_{2'}$ and $|y|=(q-1)_{2'}$ if none of $q \pm 1$ is a $2$-power, and 
$|x|=q+1$ and $|y|=q-1$ otherwise. Checking the character table 
\cite[Theorems 38.1, 38.2]{Do}, we see that the only nontrivial irreducible character $\chi$ of $G$ with $\chi(x)\chi(y) \neq 0$ is the Steinberg character $\St$, of degree $q$. Since $|\St(x)|=|\St(y)|=1$ but $|\St(z)| < |\St(1)$ for any $z \in G \smallsetminus \ZB(G)$, \eqref{sum1} holds. 
By our choice, $|u|$ and $|v|$ are coprime. We can now choose $|xy|=p$ to fulfill (B). If both $q \pm 1$ are not $2$-powers, then
we also have $(|x|,|y|,|xy|) = (|u|,|v|,|uv|)$. Since none of the maximal subgroups of 
$\SL_2(q)$ \cite[Table 8.1, 8.2]{BHR} can contain elements of every order $|x|$, $|y|$, and $p$, we conclude that $G = \langle x,y \rangle$. 

If $q = 9$, we can take $x = \begin{psmallmatrix} 1 & \alpha^4 \\ 0 & 1\end{psmallmatrix}$ and $y = \begin{psmallmatrix} 1 & 0 \\ \alpha^5 & 1\end{psmallmatrix}$, where $\alpha$ is a generator of $\F_9^\times$. Then $|x| = |y| = 3$ and $|xy| = 5$, thus proving Theorem \ref{main2} for this case. It remains to check Theorem \ref{main2} for $\PSL_2(p)$ where $p$ is a prime $\ge 5$ of the form $2^r\pm 1$. For such $p$, we wish to find generators $x,y\in\SL_2(p)$ such that $|x|,|y|,|xy|$ are all odd. Let $\ell$ be any odd prime divisor of $p^2-1$, and let $\omega\in\F_{p^2}^\times$ be an element of order $\ell$. Then by \cite[Theorem 5.2.10]{Chen24}, we may find $x,y\in\SL_2(p)$ such that $\Tr(x) = \Tr(y) = 2$ and $\Tr(xy) = \omega+\omega^{-1}$. This implies that $|x| = |y| = p$ and $|xy| = \ell$. Moreover, since $\omega+\omega^{-1}\ne 2\in\F_p$, by \cite[Lemma 5.2.9]{Chen24}, the subgroup $\langle x,y\rangle\subset\SL_2(\overline{\F_p})$ is irreducible (has no nontrivial invariant subspaces), and hence by the classification of subgroups of $\SL_2(p)$, we deduce that $x,y$ generate $\SL_2(p)$ (see \cite[Lemma 6.4.1]{Chen24}).

\smallskip
Now we consider the case $G=\SL_3(q)$, $q \geq 3$ (note that $\SL_3(2) \cong \PSL_2(7)$. Then we choose $|x|=(q^2+q+1)/\gcd(3,q-1)$, $|y|=q^2-1$
if $3 \nmid (q-1)$, and $|y|=(q^2-1)_{3'}$ if $3|(q-1)$. Note that both $|x|$ and $|y|$ are coprime to $|\ZB(G)|=d:=\gcd(3,q-1)$, and $|u|=|x|$ and 
$|v|=|y|$ are coprime.
Checking the character table of $G$ (say using {\sf Chevie}), one sees that the only nontrivial irreducible character $\chi$ of $G$ with $\chi(x)\chi(y) \neq 0$ is the Steinberg character $\St$, of degree $q^3$. Since $|\St(x)|=|\St(y)|=1$ but $|\St(z)| < |\St(1)$ for any $z \in G \smallsetminus \ZB(G)$, \eqref{sum1} holds. 
We can now choose $|xy|=p=|w|$ to fulfill (B). Assuming the contrary that $\langle x,y \rangle \leq M < G$ for a maximal subgroup $M$ of $G$, we will use \cite[Tables 8.3, 8.4]{BHR} to check possible candidates for $M$. First, if $q=4$, then $|u|=7$ and $|v|=5$, but no maximal subgroup of $S=\PSL_3(4)$ has order 
divisible by $35$, a contradiction. So we may assume $q \neq 4$, and hence $|u|$ is divisible by $\ell_1=\ppd(p,3f)$. Now $|u|=|x| \geq 13$, ruling out the possibilities  $M = d \times \PSL_2(7)$
and $M = 3 \cdot \mathsf{A}_6$. If $2|f$, so $f \neq 3$ and $|v|$ is divisible by $\ell_2=\ppd(p,2f)$, ruling the possibility $M = \SU_3(\sqrt{q})$. All other possibilities for
$M$ are ruled out by the fact that $|M|$ is divisible by $\ell_1p|y|$.

\smallskip
Next assume that $G=\SU_3(q)$, $q =p^f \geq 3$. Then we choose $|x|=(q^2-q+1)/d$, where $d:=\gcd(3,q+1)$, $|y|=q^2-1$ if $d=1$, and 
$|y|=(q^2-1)_{3'}$ if $d=3$. Note that both $|x|$ and $|y|$ are coprime to $|\ZB(G)|=d$, and $|u|=|x|$ and $|v|=|y|$ are coprime.
Checking the character table of $G$ (say using {\sf Chevie}), one sees that the only nontrivial irreducible character $\chi$ of $G$ with $\chi(x)\chi(y) \neq 0$ is the Steinberg character $\St$, of degree $q^3$. Since $|\St(x)|=|\St(y)|=1$ but $|\St(z)| < |\St(1)$ for any $z \in G \smallsetminus \ZB(G)$, \eqref{sum1} holds. 
We can now choose $|xy|=p=|w|$ to fulfill (B). Assuming the contrary that $\langle x,y \rangle \leq M < G$ for a maximal subgroup $M$ of $G$, we will use \cite[Tables 8.5, 8.6]{BHR} to check possible candidates for $M$. Note that $|u|$ is divisible by $\ell_1=\ppd(p,6f) \geq 7$, ruling out the possibilities 
$M = 3 \cdot \mathsf{A}_6$ and $3 \cdot \mathsf{A}_6 \cdot 2_3$. Next, $|v| \geq 8$, ruling out the possibility  $M = d \times \PSL_2(7)$,
and also $M=3 \cdot \mathsf{A}_7$ when $q=5$.
All other possibilities for $M$ are ruled out by the fact that $|M|$ is divisible by $\ell_1p|y|$.

\smallskip
We next record the following facts, established in \cite{GT} based on results of \cite{MSW} and \cite{LST11}.

\begin{lem}\label{rss-A}
Let $q=p^f$ and $n \geq 4$. Then the following statements hold.
\begin{enumerate}[\rm(i)]
\item {\rm (\cite[\S2.2.1]{GT})} Let $G = \SL_n(q)$ with $(n,q) \neq (6,2)$, $(7,2)$, $T_1 < G$ a cyclic maximal torus of order $(q^n-1)/(q-1)$ and $T_2 < G$ a cyclic maximal torus of order $q^{n-1}-1$. Choose $\ell_1= \ppd(p,nf)$, 
and $\ell_2 = \ppd(p,(n-1)f)$ if $(n,q) \neq (4,4)$ and $\ell_2 = 7$ otherwise. Then, for any element $x \in T_1$ of order divisible by $\ell_1$ and any
$y \in T_2$ of order divisible by $\ell_2$, $x^G \cdot y^G \supseteq G \smallsetminus \ZB(G)$. 
\item {\rm (\cite[\S2.2.2]{GT})} Let $G = \SU_n(q)$ with $(n,q) \neq (4,2)$, $T_1 < G$ a cyclic maximal torus of order $(q^n-(-1)^n)/(q+1)$ and $T_2 < G$ a cyclic maximal torus of order $q^{n-1}-(-1)^{n-1}$. If $2 \nmid n$,  choose $\ell_1 = \ppd(p,2nf)$, and 
$\ell_2 = \ppd(p,(n-1)f)$ when $n \equiv 1 \pmod{4}$, $\ell_2 = \ppd(p,(n-1)f/2)$ if $n \equiv 3 \pmod{4}$ and $(n,q) \neq (7,4)$, and 
$\ell_2 = 7$ if $(n,q) = (7,4)$. If $2|n$,  choose $\ell_2 = \ppd(p,2(n-1)f)$, and 
$\ell_1 = \ppd(p,nf)$ when $n \equiv 0 \pmod{4}$, $\ell_1 = \ppd(p,nf/2)$ if $n \equiv 2 \pmod{4}$ and $(n,q) \neq (6,4)$, and 
$\ell_2 = 7$ if $(n,q) = (6,4)$. Then, for any element $x \in T_1$ of order divisible by $\ell_1$ and any
$y \in T_2$ of order divisible by $\ell_2$, $x^G \cdot y^G \supseteq G \smallsetminus \ZB(G)$. 
\end{enumerate}
\end{lem}

\smallskip
Assume $G = \SL_4(q)$ with $q \geq 3$ (note that $\SL_4(2) \cong \mathsf{A}_8$). In the notation of Lemma \ref{rss-A}, choose $x$ of prime order
$\ell_1$, and $y$ of odd order $q^2+q+1$. Then (A) and (B) hold; in particular, we can choose $xy$ of order $p$, whence 
$(|x|,|y|,|xy|)=(|u|,|v|,|uv|)$.  Assuming the contrary that $\langle x,y \rangle \leq M < G$ for a maximal subgroup $M$ of $G$, we will use \cite[Tables 8.8, 8.9]{BHR} to check possible candidates for $M$. Note that $|y|=q^2+q+1 \geq 13$, ruling out the possibilities $\AAA_7$, $2\cdot \AAA_7$,
$\PSL_2(7)$ and $\SU_4(2)$ for $M$. All other possibilities for $M$ are ruled out by the fact that $|M|$ is divisible by $\ell_1(q^2+q+1)$.
  
\smallskip
Assume $G = \SU_4(q)$. In the notation of Lemma \ref{rss-A}, if $q \neq 3$ we choose $x$ of odd order
$(q^2+1)/\gcd(2,q-1)$, and $y$ of order $q^2-q+1$. If $q=2$ we choose $x,y$ so that $|x|= 5$ and $|y|=9$. If $q=3$ we choose $x,y$ so that $|u|= 8$ and $|v|=7$. Then (A) and (B) hold; in particular, we can choose $w=uv$ of order $p$ if $q > 3$ or $q=2$, $9$ if $q=3$. When $q=2$, this choice of $u,v,w$ shows that $\langle u,v \rangle =S$ since no maximal subgroup of 
$S$ can have elements for each order $5$ and $9$ \cite{Atlas}. 
When $q=3$, this choice of $u,v,w$ shows that $\langle u,v \rangle =S$ since no maximal subgroup of 
$S$ can have elements for each order $7$, $8$, and $9$ \cite{Atlas}.  Assuming the contrary that $\langle x,y \rangle \leq M < G$ for a maximal subgroup $M$ of $G$ when $q>3$, we will use \cite[Tables 8.10, 8.11]{BHR} to check possible candidates for $M$. Note that 
$|x| = (q^2+1)/\gcd(2,q-1) \geq 13$ and $|v|=q^2-q+1 \geq 7$, ruling out the possibilities $d \circ \SL_2(7)$ and $d \circ \Sp_4(3)$ for $M$, and also 
$M \not\cong d \circ 2 \cdot \AAA_7$ if $q \geq 5$. All other possibilities for $M$ are ruled out by the fact that $|M|$ is divisible by $(q^2+1)(q^2-q+1)/\gcd(2,q-1)$.  

\smallskip
Assume $G = \SL_5(q)$. In the notation of Lemma \ref{rss-A}, choose $x$ of prime order
$\ell_1 \geq 11$, and $y$ of order $(q^2+1)(q+1) \geq 15$. Then (A) and (B) hold; in particular, we can choose $xy$ of order $p$.  Assuming the contrary that $\langle x,y \rangle \leq M < G$ for a maximal subgroup $M$ of $G$, we will use \cite[Tables 8.18, 8.19]{BHR} to check possible candidates for $M$. Note that 
$|y| \geq 15$ shows $M \not\cong M_{11}$ when $q=3$, and $|u|=\ell_1 \geq 11$ shows that $M \neq d \times \SU_4(2)$.
In fact we have $|v|=|y| \geq 15$, and this shows $M \neq d \times \PSL_2(11)$. (We freely use the fact that maximal subgroups in class $\mathcal{S}$,
such as the ones listed in \cite[Table 8.19]{BHR}, act absolutely irreducibly on the natural $G$-module $\F_q^5$, and hence their centers are contained in
$\ZB(G)$.) All other possibilities for $M$ are ruled out by the fact that $|M|$ is divisible by $\ell_1(q^2+1)(q+1)$.

\smallskip
Assume $G = \SU_5(q)$. In the notation of Lemma \ref{rss-A}, choose $x$ of prime order
$\ell_1 = \ppd(p,10f) \geq 11$, and $y$ of order $q^2+1$ (which is coprime to $11$). 
Then (A) and (B) hold; in particular, we can choose $xy$ of order $p$ if $q>2$, and $9$ if $q=2$.  Assuming the contrary that $\langle x,y \rangle \leq M < G$ for a maximal subgroup $M$ of $G$, we will use \cite[Tables 8.20, 8.21]{BHR} to check possible candidates for $M$. Note that $|u|=\ell_1 \geq 11$ shows that $M \neq d \times \SU_4(2)$. Another candidate from \cite[Table 8.21]{BHR} is 
$M = d \times \PSL_2(11)$ with $d = \gcd(q+1,5)$, and $q=2$ or $q \geq 7$. However, in such a case $M$ does not contain any element
of order $9$, contradicting $|w|=9$ when $q=2$, and $M$ has no $11'$-element of order $>30$, contradicting $|y|=q^2+1 \geq 50$ when $q \geq 7$. 
All other possibilities for $M$ are ruled out by the fact that $|M|$ is divisible by $\ell_1(q^2+1)$.

\smallskip
Assume $G = \SL_6(q)$. In the notation of Lemma \ref{rss-A}, when $q>2$ we choose $x$ of prime order
$r=\ell_1 = \ppd(p,6f) \geq 7$, and $y$ of order $s=\ell_2= \ppd(p,5f) \geq 11$. If $q=2$, we choose $x$ of order $r=q^6-1=63$ and $y$ of order $s=\ell_2=q^5-1=31$, and set $\ell_1=7$. Then (A) and (B) hold; in particular, we can choose $xy$ of order $p$.  Assuming the contrary that $\langle x,y \rangle \leq M < G$ for a maximal subgroup $M$ of $G$, we will use \cite[Tables 8.24, 8.25]{BHR} to check possible candidates for $M$. Note that $|u|$ is divisible by the prime $\ell_1 \equiv 1 \pmod{6}$ and 
$|v| = \ell_2 \equiv 1 \pmod{5}$. This rules out all possibilities for $M$ listed in \cite[Table 8.25]{BHR}.
All other possibilities for $M$ are ruled out by the fact that $|M|$ is divisible by $rs$.

\smallskip
Assume $G = \SU_6(q)$. In the notation of Lemma \ref{rss-A}, choose $y$ of prime order
$s=\ell_2 = \ppd(p,10f) \geq 11$, and $x$ of order $r=\ell_1 = \ppd(p,3f) \geq 7$ if $q\neq 4$ and of order $q^2+q+1 = 21$ if $q=4$. 
Then (A) and (B) hold; in particular, we can choose $w=uv$ of order $p$ if $q>2$ and of order $15$ if $q=2$.  Assuming the contrary that $\langle x,y \rangle \leq M < G$ for a maximal subgroup $M$ of $G$, we will use \cite[Tables 8.26, 8.27]{BHR} to check possible candidates for $M$. Note that $|u|$ is divisible by the prime $\ell_1 \equiv 1 \pmod{6}$ and 
$|v| = \ell_2 \equiv 1 \pmod{10}$. This, together with $|uv|=15$ when $q=2$, rules out all possibilities for $M$ listed in \cite[Table 8.27]{BHR}.
All other possibilities for $M$ are ruled out by the fact that $|M|$ is divisible by $rs$. 

This proves Theorem \ref{main1} and \ref{main2} for all $\PSU_6(q)$ with $q \ne 2$. To verify Theorem \ref{main2} for $S = \PSU_6(2)$, let $G$ be the unique perfect central extension of $S$ with center $C_3$. Then $G$ admits a generating pair $x,y$ with $|x| = |y| = 11, |xy| = 7$, and which is Nielsen equivalent to the pair given by the \cite{GAP} command
$$\texttt{GeneratorsOfGroup(AtlasGroup("3.U6(2)"));}$$

\smallskip
Assume $G=\SL_7(q)$. In the notation of Lemma \ref{rss-A}, when $q> 2$ we choose $x$ of order
$r=\ell_1 = \ppd(p,7f) \geq 29$, and $y$ of order $s=\ell_2 = \ppd(p,6f) \geq 7$. 
Then (A) and (B) hold; in particular, we can choose $w=uv$ of order $2$. If $q=2$, we choose $x \in G$ of order $r=\ell_1=127$ and $y \in G$ in class $7c$ in \cite{GAP} of order $s=\ell_2=7$, and using \cite{GAP} we can again check that 
\eqref{sum1} holds; in particular, we can choose $xy$ of order $2$. According to \cite[Tables 8.37, 8.38]{BHR}, no maximal subgroup of $G$ can have 
order divisible by $L=2\ell_1\ell_2$.  It follows that $S = \langle u,v \rangle$.

\smallskip
Assume $G = \SU_7(q)$. In the notation of Lemma \ref{rss-A}, choose $x$ of order
$r=\ell_1 = \ppd(p,14f) \geq 29$, and $y$ of order $s=\ell_2 = \ppd(p,3f) \geq 7$ if $q\neq 4$ and of order $7$ if $q=4$. 
Then (A) and (B) hold; in particular, we can choose $xy$ of order $2$. According to \cite[Tables 8.37, 8.38]{BHR}, no maximal subgroup of $G$ can have 
order divisible by $2\ell_1$. Hence $S = \langle u,v \rangle$.

\subsection{The case $G=\SL_n(q)$ with $n \geq 8$} 
In the notation of Lemma \ref{rss-A} we choose $x$ of order
$r=\ell_1 = \ppd(p,nf) \geq nf+1 \geq 11$, and $y$ of order $s=\ell_2 = \ppd(p,(n-1)f) \geq (n-1)f+1 \geq 11$. 
Then (A) and (B) hold; in particular, we can choose $xy$ of order $\ell_3=\ppd(p,(n-2)f)$ if $(n,q) \neq (8,2)$ and 
order $\ell_3=7$ otherwise. We also note that $(|x|,|y|,|xy|)=(|u|,|v|,|uv|)$, and
$$\max(\ell_1,\ell_2) \geq 2n-1.$$
(Indeed, otherwise we must have that $f=1$ and $(\ell_1,\ell_2) = (n+1,n)$, which is impossible.)

\smallskip
It remains to show that $S=\langle u,v \rangle$. 
Assume the contrary: $\langle x,y\rangle$ is contained in a maximal subgroup $M$ of $\SL_n(q)$. By its choice, $x$ has 
the $\ppd(n,q;e)$-property as defined in \cite{GPPS}, with $e=n$. Similarly, $y$ has 
the $\ppd(n,q;n-1)$-property. Hence we can apply \cite[Main Theorem]{GPPS} to see that $M$ is in one 
of the Examples 2.1--2.9 described therein. 
In what follows, to indicate $M$ in as in Example 2.k of \cite{GPPS}, $1 \leq \mbox{k} \leq 9$, we will say that
$$M \mbox{ is in (2.k), or we are in (2.k)}.$$

Suppose $M$ is in (2.1). Since $\ell_1=\ppd(p,nf)$, we must have $q_0=q$. Furthermore, as $M < \SL_n(q)$, either $2|n$ and 
$M \leq \mathrm{CSp}_n(q)$ or $M \leq \mathrm{CO}^-_{n}(q)$, or $2|f$, $2 \nmid n$, and $M \leq \mathrm{CU}_n(\sqrt{q})$
(these three groups are conformal groups, i.e. the subgroup of $\GL_n(q)$ that preserves the symplectic, orthogonal, or Hermitian form,
up to $\F_q$-scalars). However, none of these conformal groups have order divisible by $\ell_2=\ppd(p,(n-1)f)$.

\smallskip
Note that $x$ is irreducible on the natural module $\F_q^n$, showing $M$ is not in (2.2). Also, $e=n$ so $M$ is not in (2.3).
Suppose $M$ is in (2.4). As $e=n$, we have $M \leq \GL_{n/b}(q^b) \cdot b$ for some divisor $b>1$ of $n$. Then $\ell_2=\ppd(p,(n-1)f)$ cannot
divide $|\GL_{n/b}(q^b)|$, so $\ell_2$ must divide $b$. But $\ell_2 \geq (n-1)f+1 \geq n$, so $b=n=\ell_2$. In this case,
$M \leq \GL_1(q^n)\cdot \ell_2$, a contradiction as $\ell_3$ divides $|xy|$.  

\smallskip
Suppose $M$ is in (2.5). Then $n=2^m$ and $\ell_1=n+1$ is a Fermat prime. As $\ell_1 \geq nf+1$, we have $q=p$. As mentioned above,
$y$ has the $\ppd(n,q;n-1)$-property, and $\ell_2 =\ppd(p,n-1) \geq 2n-1$ (as $n=2^m$). This is impossible for the groups in (2.5).

\smallskip
All the remaining candidates for $M$ are almost quasisimple; in particular, $L:=M^{(\infty)}$ is quasisimple. 
For the groups in (2.6), $\max(\ell_1,\ell_2) \leq \max(n+2,7) < 2n-1$, a contradiction. 
For the groups in (2.8), $e$ can be equal to $n$ only when $n=4$, a contradiction.

\smallskip
Suppose we are in (2.7). For a fixed pair $(n,L)$, the only pairs of
two distinct primes $\{\ell_1,\ell_2\}$ both larger than $10$ that occurs for groups in (2.7) are $\{11,23\}$ (for $n=11$, $p=2$, and 
$L = M_{23}$ or $L=M_{24}$), $\{17,19\}$ (with $n=18$), and $\{11,13\}$ (with $n=12$). The latter two pairs cannot realize in our situations
because we would then have $\{\ell_1,\ell_2\} = \{n-1,n+1\}$. So we are in the former case, whence $q=p=2$ and $L \lhd M \leq \NB_G(L)=L$, 
i.e. $M=L$ and $G=\SL_{11}(2)$. However, in this case we can choose $\ell_1=\ppd(2,11)$ to be $89$, which does not divide $|M|$, a contradiction.

\smallskip
Finally, suppose $M$ is in (2.9). Then again $11 \leq \ell_1 \in \{n+1,2n+1\}$ and $11 \leq \ell_2 \in \{n,2n-1\}$. None of the groups in \cite[Table 7]{GPPS}
can possess such a pair $\{\ell_1,\ell_2\}$. Inspecting possibilities for $\{\ell_1,\ell_2\}$ for a fixed pair
$(n,L)$ in \cite[Table 8]{GPPS}, we see that $L=\PSL_2(s)$, and either $s=2^c$, $n=s$ or $s+1$, $\{\ell_1,\ell_2\} = \{s-1,s+1\}$, 
or $s$ is a prime, $n = (s \pm 1)/2$, and $\{\ell_1,\ell_2\} = \{s,(s\pm 1)/2\}$. In the first case, $n \geq 8$ implies that $s=2^c \geq 8$, so 
$s-1$ and $s+1$ cannot be primes at the same time. In the second case, the indicated ranges for $\ell_1,\ell_2$ imply that $n=(s-1)/2$,
$\ell_1=2n+1$, and $\ell_2=n \geq 8$; in particular, $n \geq 11$. Now we observe that the element $xy$ also has the
$\ppd(n,q;n-2)$-property, as it has order divisible by $\ell_3=\ppd(p,(n-2)f)$. As $n$ is prime, $\ell_3 \geq 2n-3$, and  
there is no place for the third prime $\ell_3$ to occur for $M$.

\subsection{The case $G=\SU_n(\sqrt{q})$ with $n \geq 8$} 
We depart from the notation of \S\ref{slu-small} and denote $q=p^{2f}$, so that 
$G=\SU_n(p^f) = \SU(V)$ is a subgroup of $\GL_n(q) = \GL(V)$, 
where $V=\F_q^n$ is endowed with a non-degenerate Hermitian form $\circ$.
In the notation of Lemma \ref{rss-A} we choose $x$ of order
$r=\ell_1$, and $y$ of order $s=\ell_2$. If $2 \nmid n$, then $\ell_1=\ppd(p,2nf) \geq 2n+1 \geq 19$.
If $2|n$, then $\ell_2 = \ppd(p,2(n-1)f) \geq 2n-1 \geq 19$.
Then (A) and (B) hold; in particular, we can choose $xy$ of order $\ell_3=\ppd(p,2(n-2)f) \geq 2n-3$ if $2 \nmid n$, and 
order $\ell_3=\ppd(p,2(n-3)f) \geq 2n-5$ if $2|n$. Note that $(|x|,|y|,|xy|)=(|u|,|v|,|uv|)$.

\smallskip
It remains to show that $S=\langle u,v \rangle$. 
Assume the contrary: $\langle x,y\rangle$ is contained in a maximal subgroup $M$ of $\SL_n(q)$. When $2 \nmid n$, $x$ has 
the $\ppd(n,q;e_1)$-property with $e_1=n$, and $xy$ has the $\ppd(n,q;e_2)$-property with $e_2=n-2$. When $2|n$, $y$ has 
the $\ppd(n,q;e_1)$-property with $e_1=n-1$, and $xy$ has the $\ppd(n,q;e_2)$-property with $e_2=n-3$. 
Hence we can apply \cite[Main Theorem]{GPPS} to see that $M$ is in one 
of the Examples 2.1--2.9 described therein. 

\smallskip
Suppose $M$ is in (2.1). Since $e_1 \geq n-1$, we must have $q_0=q$. Furthermore, as $M < \SU_n(\sqrt{q})$, $2|e_1$, and either $2|n$ and 
$M \leq \mathrm{CSp}_n(q)$, or $M \leq \mathrm{CO}^\eps_{n}(q)$. However, our $e_1$ is always odd, a contradiction.

\smallskip
Suppose $M$ is in (2.2), i.e. $M$ is reducible on $V=\F_q^n$. As $x$ is irreducible on $V$ when $2 \nmid n$, we must have that
$2|n$. In this case, $y$ cannot act nontrivially on any subspace of $V$ of dimension less than $n-1$, so $M < \SU(V)$ preserves 
an orthogonal decomposition $V = \F_q^{n-1} \oplus \F_q$, showing that
$M \leq \GU_1(\sqrt{q}) \times \GU_{n-1}(\sqrt{q})$. However, the order of such $M$ is not divisible by $\ell_1=|x|$. 

\smallskip
Suppose $M$ is in (2.3). As $e_1=n$ when $2 \nmid n$, we must have that $2|n$ and $e_1=n-1$. But in this case
$\ell_2 = \ppd(p,2(n-1)f) \geq 2n-1$, impossible for groups in (2.3).

\smallskip
Suppose $M$ is in (2.4). In \cite[Example 2.4(a)]{GPPS} we have $e_1=n-1$, so $2|n$, but then $\ell_2 \geq 2n-1$, a contradiction.  
So $M \leq \GL_{n/b}(q^b) \cdot b$ for some divisor $b>1$ of $\gcd(n,e_1)$. This implies that $e_1=n$, whence $2 \nmid n$. 
Using $1 < b|n$ we can check that $\ell_2 \in \{\ppd(p,(n-1)f),\ppd(p,(n-1)f/2)\}$ cannot
divide $|\GL_{n/b}(q^b)|$, so $\ell_2$ must divide $b$. But $\ell_2 \geq (n-1)f/2+1 > n/2$, so $b=n=\ell_2$. In this case,
$M \leq \GL_1(q^n)\cdot \ell_2$, a contradiction as $\ell_3$ divides $|xy|$.  

\smallskip
Suppose $M$ is in (2.5). Then $n=2^m$ and $\max(\ell_1,\ell_2)=n+1$, a contradiction.

\smallskip
All the remaining candidates for $M$ are almost quasisimple; in particular, $L:=M^{(\infty)}$ is quasisimple. 
For the groups in (2.6), $\max(\ell_1,\ell_2) \leq \max(n+2,7) < 2n-1$, a contradiction. For the groups in (2.8), $e_1$ can be $\geq n-1$ only when $n \leq 7$, a contradiction.

\smallskip
Suppose we are in (2.7); in particular, $n \neq 8$, and so $n \geq 9$. If $2\nmid n$, then we note that $\ell_1,\ell_3 \geq 19$. Similarly,
if $2\nmid n$, then $\ell_2,\ell_3 \geq 19$. However, for a fixed pair $(n,L)$, there is no pairs of
two distinct primes both larger than $18$ that can occur for groups in (2.7), a contradiction.

\smallskip
Finally, suppose $M$ is in (2.9). None of the groups in \cite[Table 7]{GPPS}
can admit a prime $\ell_i > 13$. Inspecting possibilities for $\{\ell_1,\ell_3\}$ when $2\nmid n$, and $\{\ell_2,\ell_3\}$ when $2|n$ for a fixed pair
$(n,L)$ in \cite[Table 8]{GPPS}, we see that $L=\PSL_2(s)$, and either $s=2^c$ and $n=s$ or $s+1$, with the pair of primes being $\{s-1,s+1\}$, 
or $s$ is a prime and $n = (s \pm 1)/2$, with the pair of primes being  $\{s,(s\pm 1)/2\}$. In particular, at least one of the occurring primes 
is $\leq n+1$, whereas both of our primes in consideration is at least $n+3$, a contradiction.

\section{Proof of Theorems \ref{main1}  and \ref{main2}: Symplectic and odd-dimensional orthogonal groups}

\subsection{Some low-rank groups}

\smallskip


First assume that $G = \Sp_6(2)$. Let $H$ be its unique double cover. Using \cite{GAP}, we can check that there exist generators $x,y\in H$ such that $|x| = 5,|y| = 7,|xy| = 9$. This pair is Nielsen equivalent to the pair given by the \cite{GAP} command
$$\texttt{GeneratorsOfGroup(AtlasGroup("2.S6(2)"));}$$
This proves Theorems \ref{main1} and \ref{main2} for $\Sp_6(2)$.

\smallskip
Assume $G = \Sp_8(3)$. As shown in the proof of \cite[Lemma 2.12]{GT}, we can choose $x \in G$ of order
$\ell_1=41$, and $y$ of order $s=10$ such that $x^Gy^G = G \setminus \{1\}$. In particular, we can choose $x$ and 
$y$ so that $|xy|=13$. Checking \cite[Tables 8.48, 8.49]{BHR}, we see that no maximal subgroup of $G$ has order divisible by $41 \cdot 13 \cdot 10$, 
whence $G = \langle x,y \rangle$. Note that $v=y\ZB(G)$ also has order $10$ in $S$, so we have proved Theorem \ref{main2} for $S = \PSp_8(3)$.

\smallskip
Assume $G = \mathrm{Spin}_9(3)$. As shown in the proof of \cite[Lemma 2.12]{GT}, we can choose $x \in G$ of order
$\ell_1=41$, and $y \in G$ of order $s=39$ such that $x^Gy^G = G \setminus \ZB(G)$. In particular, we can choose $x$ and 
$y$ so that $|xy|=5$. Checking \cite[Tables 8.58, 8.59]{BHR}, we see that the only maximal subgroup of $S$ of order divisible by $41 \cdot 39$
is $\Omega^-_8(3) \cdot 2$. However, such a group has no element of order $39$ \cite{GAP}.
Hence $S = \langle u,v \rangle$, and we have proved Theorems \ref{main1} and \ref{main2} for $S = \Omega_9(3)$.

\smallskip
Next we handle the case $S = \PSp_{4}(q)'$. We may assume that $q \geq 4$, since $\Sp_4(2)' \cong \AAA_6$ and $\PSp_4(3) \cong \SU_4(2)$ have already been handled.
First we consider the case $2 \nmid q \geq 5$. The proof of \cite[Lemma 2.11]{GT} shows that we can find a regular semisimple element $x \in G$ of 
odd order $r=(q^2+1)/2 \geq 13$, and a regular semisimple element $y \in G$ of order $s|(q^2-1)/2$ such that $x^Gy^G \supseteq G \setminus \ZB(G)$. 
In fact, unless $q = 5,7$, $y$ can be chosen so that $s$ is odd: this was done in the proof of \cite[Lemma 2.11]{GT} if both $q \pm 1$ are not $2$-powers or $q=9$.
If $q \geq 17$ is a Fermat prime, we choose $y$ from class $B_4(2,4)$ of order $(q+1)/2$ in the notation of \cite{Sr}, and if $q \geq 31$ is a Mersenne prime, we choose $y$ from class $B_3(2,4)$ of order $(q-1)/2$.
Hence we can conjugate $x$ and $y$ so that $xy$ is a transvection (of order $p$); furthermore,
$(|x|,|y|,|xy|)=(|u|,|v|,|uv|)$ if $q \geq 9$.
Now suppose that $\langle x,y \rangle \leq M$ for some maximal subgroup $M$ of $G$.  Note that
$r=|x|$ is odd and divisible by $\ppd(p,4f)$. So using \cite[Tables 8.12, 8.13]{BHR}, we get that $M = \Sp_2(q^2) \rtimes 2$, an extension field subgroup of
$G$. As $|xy|=p$ is odd, $xy \in \Sp_2(q^2)$. But any element of order $p$ in $\Sp_2(q^2)$ while acting on $\F_{q^2}^2$ a fixed point subspace of 
dimension $1$ over $\F_{p^2}$, i.e. $2$-dimensional fixed point subspace on $\F_q^4$, and thus cannot be a transvection. 
Hence $G=\langle x,y \rangle$, as desired. We now check Theorem \ref{main2} for $q = 5,7$. For $\PSp_4(5)$, we check using \cite{GAP} that its unique double cover admits a generating pair $x,y$ such that $|x| = 5, |y| = 13, |xy| = 15$ and which is Nielsen equivalent to the pair
$$\texttt{GeneratorsOfGroup(AtlasGroup("2.S4(5)"));}$$
For $q = 7$, its unique double cover admits a generating pair $x,y$ with $|x| = |y| = |xy| = 25$, and which is Nielsen equivalent to the pair
$$\texttt{GeneratorsOfGroup(AtlasGroup("2.S4(7)"));}$$

Suppose now that $q=2^f \geq 4$. Then we choose a regular semisimple element $x$ of order $r=q^2+1$, in class $B_5(i)$ in the notation of
\cite{Eno}, and a regular semisimple element $y$ of order $s=q+1$, in class $B_4(i,j)$ and with centralizer of order
$(q+1)^2$. Using the character table of $G$ \cite{Eno}, we can check
that only three characters: $1_G$, the Steinberg character $\St$, and another ($\theta_8$ in \cite{Eno}, of degree $q(q-1)^2/2$) can be non-vanishing 
at $x$ and at $y$. This allows us to verify \eqref{sum1}, and thus $x^Gy^G \supseteq G \setminus \{1\}$. As in the previous case, we may choose
$x$ and $y$ so that $xy$ is a transvection (of order $2$0; by our choice $(|x|,|y|,|xy|)=(|u|,|v|,|uv|)$.
Now suppose that $\langle x,y \rangle \leq M$ for some maximal subgroup $M$ of $G$.  Note that
$r=|x|$ is divisible by $\ppd(2,4f)$. So using \cite[Table 8.14]{BHR}, we get that $M \cong \Sp_2(q^2) \rtimes 2$ or $\SO^-_4(q) \cong \Sp_2(q^2)$. 
In either case, as $y$ has odd order (dividing $q^2-1$), $y \in \Sp_2(q^2)$ and hence its centralizer in $M$ has order divisible by $q^2-1$, contradicting the fact
that $|\CB_G(y)|=(q+1)^2$. Hence $G=\langle x,y \rangle$, as desired.

\subsection{The general case}
We will need the following result which is essentially recorded in \cite[\S2.2.3]{GT}, whose proof relies on \cite[Theorem 2.3]{MSW}.

\begin{lem}\label{rss-BC}
Let $q=p^f$, $n \geq 2$, $(n,q) \neq (2,2)$, $(2,3)$, $(3,2)$, $(4,3)$. Let $G = \mathrm{Spin}_{2n+1}(q)$ with $2 \nmid q$, or $G=\Sp_{2n}(q)$, and let
$T_1 < G$ be a maximal torus of order $q^n+1$, $T_2 < G$ a cyclic maximal torus of order $q^n-1$. Let $\ell_1 = \ppd(p,2nf)$.
\begin{enumerate}[\rm(i)]
\item Suppose $2 \nmid n$. Choose $\ell_2 = \ppd(p,nf)$ if $(n,q) \neq (3,4)$ and $\ell_2 = 7$ otherwise. Then, for any elements 
$x \in T_1$ of order divisible by $\ell_1$ and $y \in T_2$ of order divisible by $\ell_2$ we have $x^G \cdot y^G \supseteq G \smallsetminus \ZB(G)$. 
\item Suppose $2|n \geq 4$. Choose $s_1 = \ell_2=\ppd(p,nf)$ if  $(n,q) \neq (6,2)$ and $s_1=3$ otherwise. Furthermore, if $n \geq 6$, choose 
$s_2=\ppd(p,nf/2)$ if  $(n,q) \neq (12,2)$ and $s_2=7$ otherwise. If $n = 4$, choose 
$s_2=\ppd(p,nf/2)$ if  $q$ is not a Mersenne prime, and $s_2=3$ if $q \geq 7$ is a Mersenne prime. Then, for any elements 
$x \in T_1$ of order divisible by $\ell_1$ and $y \in T_2$ of order divisible by $s_1s_2$ we have $x^G \cdot y^G \supseteq G \smallsetminus \ZB(G)$. 
\end{enumerate}
\end{lem}

First we consider the case where $G = \mathrm{Spin}_7(q)$ with $2 \nmid q$, or $G = \Sp_6(q)$ with $q >2$. Then for the elements $x,y \in G$ constructed in 
Lemma \ref{rss-BC}(i) we have $x^Gy^G = G \setminus \ZB(G)$. In particular, we can choose $x$ and 
$y$ so that $xy$ has order $(q^2+1)/\gcd(2,q-1)$ (indeed, $G/\ZB(G)$ contains a subgroup isomorphic to $\PSp_4(q) > \PSp_2(q^2)$).  
Assuming the contrary that $\langle x,y \rangle \leq M < G$ for a maximal subgroup $M$ of $G$, we will use \cite[Tables 8.39, 8.40]{BHR},
respectively \cite[Tables 8.28, 8.29]{BHR} to check possible candidates for $M$. Note that $|xy|$ is odd and $\geq 25$ if $q \geq 7$, 
$|xy|=5$ if $q=3$, and $\ell_2 = 31$ if $q = 5$. This allows us to show $M \not\cong \SL_2(13)$ when $2 \nmid q$. All other candidates for 
$M$ are ruled out by considering the orders of $x$, $y$, and $xy$, using in particular that $\max(\ell_1,\ell_2) \geq 13$.

\medskip
Now we assume $n \geq 4$, $(n,q) \neq (4,3)$,  and handle the cases where $G =\Omega_{2n+1}(q)$ with $2 \nmid q$, or
$G = \Sp_{2n}(q)$. We will construct the elements $x$ and $y$ using Lemma \ref{rss-BC} (but working in $\Omega_{2n+1}(q)$ instead of
$\mathrm{Spin}_{2n+1}(q)$ in the former case).  If $2 \nmid n$, then we choose 
$x \in T_1$ of order $r=\ell_1$ and $y \in T_2$ of order $s=\ell_2$ as in Lemma \ref{rss-BC}(i). If $2|n$, then we choose 
$x \in T_1$ of order $r$ the part of $q^n+1$ that is coprime to $2n$  (so divisible by $\ell_1$ and coprime to $2n(q-1)$), and $y \in T_2$ of order $s=s_1s_2$ 
as in Lemma \ref{rss-BC}(ii). By our construction, 
\begin{equation}\label{r10}
  \begin{aligned}
  & \ell_1 = \ppd(p,2nf) \geq 2nf+1 \geq 11,\\ & \ell_2 = \ppd(p,nf) \geq nf+1 \mbox{ if }(n,q) \neq (6,2),\\
  & \max(\ell_1,\ell_2) \geq 4n+1 \geq 23 \mbox{ if }2 \nmid n,\\
  & r \geq 4n+1 \geq 17 \mbox{ unless }(n,q)=(9,2). \end{aligned}
\end{equation}
(Indeed, the first two claims are given by the choice of $\ell_i$. For the third claim, note that, since $2nf|(\ell_1-1)$, either $\ell_1 \geq 4n+1 \geq 17$, in which case
we are done, or $\ell_1=2n+1$. Consider the latter case. If $2 \nmid n$, then $n|(\ell_2-1)$ but $\ell_2 \neq n+1$, $3n+1$ (by parity) and $\ell_2 \neq \ell_1$, so 
$\ell_2 \geq 4n+1$. Also, by \cite{F}, we have $r \geq 4n+1$, unless $(n,f) = (5,2)$, $(6,2)$, and $(9,2)$, in which cases we have $r=33$, $65$,
and $19$, respectively.)

 Let $V = \F_q^{d}$ with $d=2n+1$, respectively $d=2n$, be the natural module for $G$, considered with the corresponding 
bilinear form $(\cdot,\cdot)$. As $x^Gy^G  \supseteq G \setminus \ZB(G)$,
we can choose $x$ and $y$ such that $xy$ is a regular unipotent element, so that it acts with a single Jordan block of size $d$ and a one-dimensional
fixed point subspace $\langle v_1 \rangle_{\F_q}$ on $V$; in particular, $(v_1,v_1)=0$. 

\smallskip
It remains to show that $S=\langle u,v \rangle$. 
Assume the contrary: $H:=\langle x,y\rangle$ is contained in a maximal subgroup $M$ of $\SL_d(q)$. By its choice, $x$ has 
the $\ppd(d,q;e)$-property, with $e=d$, respectively $e=d-1$. Hence we can apply \cite[Main Theorem]{GPPS} to see that $M$ is in one 
of the Examples 2.1--2.9 described therein. 

\smallskip
Suppose $M$ is in (2.1). Since $\ell_1=\ppd(p,ef)$, we must have $q_0=q$. Certainly, $M$ cannot contain $\SL_d(q)$. If $M \rhd \Sp_d(q)$, then
$G=\Sp_d(q)$ and hence $M=G$, a contradiction. Suppose $M \rhd \Omega^\eps_d(q)$. Then, since $M < G$, we must have
that $G=\Sp_d(q)$ with $d=2n$, $2|q$, and  $M=\GO^-_{d}(q) = \Omega^-_{d}(q) \cdot 2$. Note that both $x$ and $y$ have odd orders,
so $H=\OB^2(H) \leq \OB^2(M) = \Omega^-_{d}(q)$. But then the element $xy$ which has quasi-determinant $-1$ cannot lie inside
$\Omega^-_{12}(q)$, again a contradiction.
The remaining possibility is that
$\SU_d(\sqrt{q}) \lhd M \leq \mathrm{CU}_d(\sqrt{q})$ if $2|f$. However, if $2 \nmid d$, then $d=2n+1$, and $\ppd(p,(2n+1)f)$ divides
$|\SU_d(\sqrt{q})|$ but not $|G|$. If $2 \mid d$, then $d=2n$, and $\ell_1=\ppd(p,2nf)$ divides $|M|$ but not 
$|\mathrm{CU}_d(\sqrt{q})|$. 

\smallskip
Suppose $M$ is in (2.2), i.e. $M$ is reducible on $V$. If $d=2n$, then $x$ is irreducible on $V$. Hence $d=2n+1$, in which case 
the $\langle x \rangle$-module $V$ is an orthogonal sum of two irreducible submodules: a non-degenerate subspace $\langle v_2 \rangle_{\F_q}$ and its $2n$-dimensional orthogonal complement. It follows that $M$ fixes $\langle v_2 \rangle_{\F_q}$, and hence $xy$ fixes $v_2$. But this is 
a contradiction, since the fixed point subspace for $xy$ is spanned by the singular vector $e_1$. 

\smallskip
Suppose $M$ is in (2.3). Then $\ell_1 = e+1 \leq d$, whence $d=2n+1=\ell_1$, $2 \nmid q$, and $M \leq \GL_1(q) \wr \SSS_d$. Now if 
$2 \nmid n$, then $\ell_2 > d$ by \eqref{r10} and $\ell_2 \nmid (q-1)$, and so
$\ell_2 \nmid |M|$,  a contradiction. Suppose $2|n$, whence $n \geq 6$ (as $\ell_1=2n+1$). Now $d=\ell_1$ is prime, and the only element order divisible
by $\ell_1$ in $\SSS_d$ is $\ell_1$, so $x^{\ell_1} \in \GL_1(q)^d$ and thus $x^{\ell_1(q-1)}=1$. But $|x|=r$ is coprime to $q-1$, so $x^{\ell_1}=1$ and thus $r=\ell_1=2n+1$, contrary to \eqref{r10}.

\smallskip
Suppose $M$ is in (2.4), so that $M \leq \GL_{d/b}(q^b) \cdot b$ for some $1 < b|d$. Here, the first possibility is that $\ell_1=d=b=e+1$, whence 
$d=2n+1$, $2 \nmid q$, and $M \leq \GL_1(q^d) \cdot d$. On the other hand, $s=|y|$ has a prime divisor $\ell_2 \neq \ell_1$ that divides
$q^n-1$ and so coprime to $q^d-1$, a contradiction. Hence we are in the second possibility: $b>1$ divides $\gcd(e,d)$. As $d-1 \leq e \leq d$,
this implies that $e=d=2n$. We claim that in this case 
\begin{equation}\label{xy1}
  xy \in \GL_{d/b}(q^b).
\end{equation}  
This is certainly the case if $|x|$ and $|y|$ are both coprime to $b$,
as this ensures that $H \leq \GL_{d/b}(q^b)$.
Now, by our choice, $|x|=r$ is coprime to $d$. If $2 \nmid n$, then $|y|=\ell_2 > n$ is again coprime to $d$. If $2|n$, then the only 
case that $|y|=s$ is not coprime to $d$ is when $(n,q,s)=(6,2,21)$. In this exception, using \cite[Table 8.80]{BHR}, we can check that 
$M = \Sp_6(4) \cdot 2$ (in which case $H = \OB^2(H) \leq \Sp_6(4) < \GL_{d/b}(q^b)$), or 
$M = \Sp_4(8) \cdot  3$ (in which case the $2$-element $xy$ belongs to $\Sp_4(8) < \GL_{d/b}(q^b)$). We have therefore proved \eqref{xy1},
which implies that the $p$-element $xy$ acts on $V=\F_q^d$ with every Jordan block repeating $b$ times. But this is a contradiction since 
$xy$ is regular unipotent.

\smallskip
Suppose $M$ is in (2.5). Then $2n=d=2^{m+1}$ with $m \geq 2$. Furthermore, $p>2$, and  $\ell_1=d+1$ is a Fermat prime, which implies 
that $m \geq 3$ is odd,  whence $d-1=(2^{(m+1)/2}-1)(2^{(m+1)/2}+1)$ and $3|(2^m+1)$.
The structure of $M$ then shows that any odd prime divisor of $|M|$ is either $\ell_1$, a divisor of $q-1$, or 
at most $2^m-1=n-1$. On the other hand, $|y|$ has a prime divisor 
$\ell_2 = \ppd(p,nf) \geq n+1$, a contradiction.

\medskip
All the remaining candidates for $M$ are almost quasisimple; in particular, $L:=M^{(\infty)}$ is quasisimple acting absolutely irreducibly on $V$.
Suppose $M$ is in (2.6). Since $\ell_1 \geq 11$ by \eqref{r10}, we have that $L = \AAA_m$ acting on $V$ via its deleted permutation module
and $M \leq (C_{q-1} \times L) \cdot 2$.
Here, $d \in \{m-2,m-1\}$ and all prime divisors of $M$ that are coprime to $q-1$ are at most $m$. In particular, 
$\ell_1 \leq m \leq d+2 \leq 2n+3$. Using \eqref{r10}, we deduce that $\ell_1=2n+1$ and $f=1$; in particular, $n > 4$. 
Now, if $2 \nmid n$,  then $\ell_2 \geq 4n+1$ by \eqref{r10} and hence cannot divide $|M|$. Hence $2|n \geq 6$. Note that the only element order divisible 
by $\ell_1$ in $\SSS_m$ is $\ell_1$. It follows that $x^{\ell_1} \in C_{q-1}$, and hence $x^{\ell_1(q-1)}=1$. But $(r,q-1)=1$, so $x^{\ell_1}=1$ and 
hence $r=\ell_1=2n+1$, which contradicts \eqref{r10}.

\smallskip
Suppose $M$ is in (2.7). 
Since $d-1 \leq e=2n \geq 8$, from \cite[Table  5]{GPPS} we see that $(e,d,\ell_1) = (10,10,11)$, $(10,11,11)$, $(12,12,13)$, $(18,18,19)$, 
$(22,22,23)$, $(22,23,23)$, or $(28,28,29)$. For $e=10$, i.e. $n=5$, we have $\ell_2 \geq 21$ and $5|(\ell_2-1)$ by \eqref{r10}, so $\ell_2 \geq 31$. 
For $e=18$, i.e. $n=9$, we have $\ell_2 \geq 37$ by \eqref{r10}. For $e=22$, i.e. $n=11$, we have $\ell_2 \geq 45$ and $11|(\ell_2-1)$ by \eqref{r10}, so $\ell_2 \geq 67$. In all these three cases, $\ell_2 \nmid (q-1)$ by its definition, and we can then check that $\ell_2$ does not divide $|M|$ for any of the groups in
\cite[Table 5]{GPPS}. So $d=e=12$ or $28$. If moreover $e=28$, i.e. $n=14$, then as $14|(\ell_2-1)$ and $\ell_2 \neq 29$, we have $\ell_2 \geq 43$ and hence cannot divide $|M|$. In the case $d=e=12$ we have $L= 6 \cdot \mathrm{Suz}$ and $p>2$, so $|y|=s=s_1s_2$ with 
$s_1=\ppd(p,6f) \geq 7$ and $s_2=\ppd(p,3f) \geq 7$, both different from $\ell_1=\ppd(p,12f)$. It follows that at least one of $s_1,s_2$ is 
$\geq 19$ and hence does not divide $|M$, again a contradiction.

\smallskip
Suppose we are in (2.8). Since $e =2n \geq d-1$, we see from \cite[Table 6]{GPPS} that $e \leq 6$, which is impossible since $n \geq 4$.

\smallskip
Finally, suppose $M$ is in (2.9). Since $d-1 \leq e=2n \geq 8$, for the groups in 
\cite[Table  7]{GPPS} we see that $(e,d,\ell_1) = (12,12,13)$, $p>2$, and $L=2 \cdot G_2(4)$.
As in the previous case, we now see that at least one of $s_1,s_2$ is 
$\geq 19$ and hence does not divide $|M$, a contradiction.

Next we consider the possibilities for $L$ listed in \cite[Table 8]{GPPS}. We will use the upper bounds for $\meo(L)$, the maximum order of elements 
$h \in \Aut(L)$, listed in \cite[Table 3]{GMPS}. The construction of our elements $x$ and $y$ ensures that their order is the same order in 
$\PGL(V)$, and hence in $\Aut(L)$ as $L$ acts absolutely irreducibly on $V$.
The first case is $L/\ZB(L)=\PSL_m(t)$ with $m \geq 3$ a prime and 
$e+1=\ell_1= (t^m-1)/(t-1)$.  Here, $\meo(L)=\ell_1=2n+1$, so by \eqref{r10} we must have $r=\ell_1$, and hence $(n,q) = (9,2)$ and 
$(t^m-t)/(t-1)=18$, which is impossible. 

The second case is $L/\ZB(L)=\PSU_m(t)$ with $m \geq 3$ a prime and 
$e+1=\ell_1= (t^m+1)/(t+1)$. Note that $e=(t^m-t)/(t+1) \neq 18$. Now, if $L \neq \SU_5(2)$, then $\meo(L) < 2\ell_1$ by \cite{GMPS}, and so $r=\ell_1$, 
contrary to \eqref{r10}. If $L=\SU_5(2)$, then $n=5$ and $|y|=\ell_2=\ppd(p,5f)$ is $\geq 31$ (as $\ell_2 \neq \ell_1=11)$, and
so $\ell_2 \nmid |M|$, again a contradiction. 

The third family is $L/\ZB(L)=\PSp_{2m}(t)$, where either $\ell_1=(t^m+1)/2$ and $2 \nmid t$, or $\ell_1=(3^m-1)/2$, $t=3$, and $m$ is a prime. 
It is easy to see that $t=t_0^{2^a}$ for a prime $t_0$ in the first case. As $H=\OB^2(H)$, in either case we have $H \leq \OB^2(M) \leq \ZB(M)L$
(as $\Out(L)$ is a $2$-group). Now we can check that the element $x$ of order $\ell_1$ generates a maximal torus in $L/\ZB(L)$,
and hence $r=\ell_1=e+1=2n+1$, which implies 
$(n,q) = (9,2)$ by \eqref{r10}. As $18=2n=e=\ell_1-1$, this cannot happen for the indicated $\ell_1=(t^m \pm 1)/2$.

In the remaining families we have $L/\ZB(L)= \PSL_2(t)$ with $t \geq 7$. Since $\ell_1 \geq 11$ by \eqref{r10}, in fact we have $t \geq 11$, and 
$\meo(L) = t+1$ by \cite{GMPS}. If $d \in \{t,t\pm 1\}$, then $\meo(L) \leq d+2 \leq 2n+3$, whereas when $2 \nmid n$ some $\ell_i \geq 4n+1$ by \eqref{r10}, a contradiction. If $2|n$, then $r \geq 4n+1$ by \eqref{r10}, a contradiction as well. Thus $d = (t \pm 1)/2$, and so
$\meo(L) = t+1 \leq 2d+2 \leq 4n+4$. If $(n,q) = (9,2)$ in addition, then we can choose $\ell_2 = 73$ which then does not divide $|M|$. So $(n,q) \neq (9,2)$, 
whence $r \geq 4n+1$ by \eqref{r10}. As $e=2n$ and $r$ is an odd multiple of $\ell_1$ which is $2n+1$,
$4n+1$ or $\geq 6n+1$, it follows that $r=\ell_1 =4n+1=2e+1$. Applying \cite[Theorem 3.2.2]{Tr}, we see that
$(q,n)$ is one of $(2,4)$, $(2,5)$, $(2,6)$, $(2,9)$ (which is already excluded), $(2,10)$, or $(3,9)$. However, by our choice $r$ is $33$, respectively $65$, and $4921$, if $(q,n)=(2,5)$, $(2,6)$, or $(3,9)$, respectively. If $(n,q) = (2,10)$, then $s=351$ by our construction. The only remaining case is $(n,q) = (4,2)$, 
i.e. $G = \Sp_8(2)$, $r=\ell_1=17$,  $s=15$, and $|xy|=8$. Checking the maximal subgroups with socle $\PSL_2(t)$ of $G$ \cite{GAP}, we see that 
$M = \PSL_2(17)$ whose order is however coprime to $5$, the final contradiction.

\section{Proof of Theorems \ref{main1}  and \ref{main2}: Even-dimensional orthogonal groups}
\subsection{The rank $4$ groups}

Assume $G = \mathrm{Spin}^+_8(q)$ with $q \geq 4$. By \cite[Lemma 2.4]{GT}, we can choose $x$ of prime order
$r=\ell_1=\ppd(p,6f)$, and $y$ of prime order $s=\ell_2=\ppd(q,3)$ such that $x^Gy^G = G \setminus \ZB(G)$. In particular, we can choose $x$ and 
$y$ so that $xy$ has order $(q^4-1)_{2'}$ if both $q \pm 1$ are not $2$-powers, and $(q^4-1)/\gcd(2,q-1)$ otherwise 
(indeed, $\SO^+_8(q) > \SO^+_2(q^4) \cong C_{q^4-1}$, and 
$\Omega^+_8(q)$ has index $\gcd(2,q-1)$ in $\SO^+_8(q)$).  Assuming the contrary that $\langle u,v \rangle \leq M < S=G/\ZB(G)$ for a maximal subgroup $M$ of $S$, we will use \cite[Table 8.50]{BHR} to check possible candidates for $M$. Note that $\max(\ell_1,\ell_2) \geq 13$, in fact
$\ell_2 \geq 31$ if $q=5$, and this rules out all  possibilities of $M$ with composition factor $\AAA_{8,9,10}$, $\Omega^+_8(2)$, or
$\tw2 B_2(8)$. Using the fact that $|M|$ is divisible by $\ell_1\ell_2$, we can rule out all other possibilities for $M$, except
for $M \cong \Omega_7(q)$ if $2 \nmid q$, and $M \cong \Sp_6(q)$ or $G_2(q)$ when $2|q$. However, none of these two groups can contain elements of
order equal to $|uv|$ (because $|uv| \geq (q^4-1)/\gcd(2,q-1)^2 > (q^2+1)(q+1)/\gcd(2,q-1)$), a contradiction. 

This proves Theorem \ref{main1} for
$S = P\Omega^+_8(q)$ for $q \geq 4$, and also Theorem \ref{main2} if in addition none of $q \pm 1$ is a $2$-power. To prove 
Theorem \ref{main2} for the remaining values of $q \geq 5$, we work with $L = \Omega^+_8(q) = \Omega(V)$ where $V = \F_q^8$. We again choose 
$x \in L$ of order $r=\ell_1=\ppd(p,6f)$ as before, but for $y$ we choose a regular semisimple element of order $s=(q^2+1)/2 \geq 13$ in the torus $T_2$ 
considered in \cite[\S2.5]{MSW}; in particular, $y$ has no nonzero fixed point on $V$ (such a $y$ exists since $q \geq 5$). 
As shown in the proof of \cite[Theorem 2.7]{MSW}, there are only three characters of $L$ that are non-vanishing 
at both $x$ and $y$: $1_L$, $\St_L$ of degree $q^{12}$, and another one, $\gamma$, of degree $q^3(q^3-1)(q-1)^3/2$. They all take values $\pm 1$
at $x$ and $y$, except that $|\gamma(y)|=2$. Now we choose $z \in L$ of order $(q^2+1)/2$ which has a $4$-dimensional fixed point subspace on
$V$, so that $|\CB_L(y)|$ divides $q^2(q^2+1)(q^4-1)$, which implies $|\gamma(z)| < q^2(q^2+1)$ and $|\St_L(z)| = q^2$. It follows from \eqref{sum1}
that $z \in x^Ly^L$, so without loss we may assume that $z=xy$. Again using \cite[Table 8.50]{BHR}, we can check that any maximal subgroup of 
$L$ that contains elements of both orders $r$ and $s$ contains a unique conjugacy class of cyclic subgroups of order $s$. However, $\langle y \rangle$
and $\langle xy \rangle$ are not conjugate in $L$. It follows that $L=\langle x,y \rangle$. Since $S = L/\ZB(L)$ with $|\ZB(L)|=2$, this proves 
Theorem \ref{main2} for $S$. 

\smallskip
Assume $G = 2 \cdot \Omega^+_8(2)$. Choosing $x$ from class $7a$ (number $51$) of order
$r=7$, and $y$ from class $12e$ (number $66$), and $z$ from class $5b$ (number $29$) in the notation \cite{GAP}, we can check that 
$z \in x^Gy^G$. Conjugating $x$ and $y$ suitably, we may assume that $z=xy$. We also note that 
only three characters of $S = G/\ZB(G)$ (of degree $1$, $50$, $300$) can be non-vanishing at both $x$ and $y$, and the sum in \eqref{sum1} has absolute value at
most $18/50+60/300 < 1$. Hence $u^Sv^S = S \setminus \{1\}$ for $u=x\ZB(G)$ and $v=y\ZB(G)$. As elements of $S$, $u$ belongs to class $7a$,
$v$ belongs to class $12a$, and $uv$ belongs to class $5b$ in the notation of \cite{GAP} (so $|uv|=5$).  Assume the contrary that $\langle u,v \rangle \leq M < S$ for a maximal subgroup $M$ of $S$. As $M$ contains $u$ of order $7$ and $v$ from class $12a$, $M$ is conjugate to one of three maximal subgroups of 
type $2^6 \rtimes \AAA_8$. However, none 
of these three subgroups can intersect both classes $12a$ and $5b$ at the same time, as can be seen using \cite{GAP}. Hence 
$S=\langle u,v \rangle$, and we have proved Theorems \ref{main1} and \ref{main2} for $S$.

\smallskip
Assume $S = P\Omega^+_8(3)$. Choosing $u$ of order
$r=13$ from class $13a$ in the notation of \cite{GAP}, and $v$ of order $s=5$ from class $5a$ (number $24$) in the notation \cite{GAP}, we can check that 
only theee characters (of degree $1$, $24192$, $3^{12}$) can be non-vanishing at both $u$ and $v$, and the sum in \eqref{sum1} has absolute value at
most $1/81+2/7<1$. Hence $u^Sv^S = S \setminus \{1\}$. In particular, we can choose $u$ and 
$v$ so that $uv$ belongs to class $18b$ (number $107$) in the notation of \cite{GAP} (so $|uv|=18$).  Assume the contrary that $\langle u,v \rangle \leq M < S$ for a maximal subgroup $M$ of $G$. As $M$ contains $u$ of order $13$ and $uv$ of order $18$, $M$ is conjugate to one of six maximal subgroups of 
type $\Omega_7(3)$. However, none 
of these six subgroups can intersect both classes $5a$ and $18b$ at the same time, as can be seen using \cite{GAP}. Thus we have proved Theorem \ref{main1} 
for $S = P\Omega^+_8(3)$. To prove Theorem \ref{main2} for $S$, we work inside $L=\Omega^+_8(3)$, and choose $x \in L$ from class $13a$ (number 159) and $y \in L$ from class $5a$ (number 41) in the notation of \cite{GAP}, so that $u'=x\ZB(L)$ and $v' =y\ZB(L)$ are conjugate in $S$ to the aforementioned elements $u,v$ of $S=L/\ZB(L)$. Now choosing $z \in L$ from class $5b$ (number 43) and using \cite{GAP}, we can check that $z \in x'^Ly^L$, whence
we may assume $z=xy$. We can also check using \cite{Atlas} that if $M < S$ is a maximal subgroup that contains elements of both orders $13$ and $5$ then 
$M$ has a unique conjugacy class of Sylow $5$-subgroups (of order $5$). However, by our choice $\langle y\ZB(L) \rangle$ and $\langle xy\ZB(L) \rangle$ are not 
conjugate in $S$. It follows that $L = \langle x,y \rangle$, and hence we have proved Theorem \ref{main2} for $S$.

\smallskip
Assume $G = \Omega^-_8(2)$. Choosing $x$ of order
$\ell_1=17$, and $y$ of order $\ell_2=7$ and using \cite{GAP}, we can check that $x^Gy^G = G \setminus \{1\}$. In particular, we can choose $x$ and 
$y$ so that $|xy|=2$. Since no maximal subgroup of $G$ has order divisible by $17 \cdot 7$ \cite{Atlas}, $G = \langle x,y \rangle$.

\smallskip
Assume $G = \mathrm{Spin}^-_8(q)$ with $q \geq 3$. By \cite[\S2.2.4]{GT}, we can choose $x$ of prime order
$r=\ell_1=\ppd(p,8f)$, and $y$ of prime order $s=\ell_2=\ppd(p,6f)$ such that $x^Gy^G = G \setminus \ZB(G)$. In particular, we can choose $x$ and 
$y$ so that $xy$ has order $\ell_3=\ppd(q,3)$.  According to \cite[Tables 8.52, 8.53]{BHR}, no maximal subgroup of 
$S = G/\ZB(G)$ can have order divisible by $\ell_1\ell_2\ell_3$. Hence $G = \langle x,y \rangle$. 

\subsection{The general case}
We will need the following result recorded in \cite{GT}.

\begin{lem}{\rm (\cite[\S2.2.4 and Theorem 2.7]{GT})} \label{rss-D}
Let $q=p^f$ and $n \geq 5$. 
\begin{enumerate}[\rm(i)]
\item Suppose $2 \nmid n$.  Let $G = \mathrm{Spin}^+_{2n}(q)$, 
$T_1 < G$ be a maximal torus of order $(q^{n-1}+1)(q+1)$, $T_2 < G$ a maximal torus of order $q^n-1$. Then for any elements 
$x \in T_1$ of order divisible by $\ell_1=\ppd(p,2(n-1)f)$ and $y \in T_2$ of order divisible by $\ell_2=\ppd(p,nf)$ we have $x^G \cdot y^G \supseteq G \smallsetminus \ZB(G)$. 
\item Suppose $2|n$.  Let $G = \mathrm{Spin}^+_{2n}(q)$, 
$T_1 < G$ be a maximal torus of order $(q^{n-1}+1)(q+1)$, $T_2 < G$ a maximal torus of order $(q^{n-1}-1)(q-1)$. Then for any elements 
$x \in T_1$ of order divisible by $\ell_1=\ppd(p,2(n-1)f)$ and $y \in T_2$ of order divisible by $\ell_2=\ppd(p,(n-1)f)$ if $(n,q) \neq (4,4)$ and $\ell_2 =7$ otherwise, 
we have $x^G \cdot y^G \supseteq G \smallsetminus \ZB(G)$. 
\item Let $G = \mathrm{Spin}^-_{2n}(q)$, 
$T_1 < G$ be a maximal torus of order $q^n+1$, $T_2 < G$ a maximal torus of order $(q^{n-1}+1)(q-1)$. Then for any elements 
$x \in T_1$ of order divisible by $\ell_1=\ppd(p,2nf)$ and $y \in T_2$ of order divisible by $\ell_2=\ppd(p,2(n-1)f)$ we have $x^G \cdot y^G \supseteq G \smallsetminus \ZB(G)$. 
\end{enumerate}
\end{lem}

Now we assume 
$$n \geq 5,$$  
and complete the case where $G =\Omega(V) = \Omega^\eps_{2n}(q)$, where 
$V = \F_q^{2n}$ is endowed with a $G$-invariant non-degenerate quadratic form $\QBS$ of type $\eps = \pm$. We will construct the elements $x,y \in G$ 
using Lemma \ref{rss-D} but working in $\Omega^\eps_{2n}(q)$ instead of
$\mathrm{Spin}_{2n}^\eps(q)$. 
If $2|q$, $\eps=+$, and $2|n$, we choose $x$ of order $r=q^{n-1}+1$ (so divisible by $\ell_1$), and $y$ of order $s=\ell_2$. 
If $q=2$, $\eps=-$, and $n=6$, we choose $x$ of order $r=q^n+1 = 65$ (so divisible by $\ell_1=13$), and $y$ of order $s=\ell_2$.
In all other cases,
we choose $x$ of order $r$, the (full) $\ell_1$-part of $q^n+1$ when $\eps=-$ and of $q^{n-1}+1$ when $\eps=+$,  and $y$ of order $s=\ell_2$.  
By our construction, 
\begin{equation}\label{r20}
  \begin{aligned}
  & \ell_1 \geq 2nf+1 \geq 11 \mbox{ if }\eps=-, & \ell_1 \geq (2n-2)f+1 \geq 11 \mbox{ if } \eps=+.
\end{aligned}
\end{equation}
Moreover, 
\begin{equation}\label{r21}
  \begin{array}{lll}
   \mbox{If }\eps=+, & \mbox{then } r \geq 6(n-1)+1, & \mbox{unless }q=2 \mbox{ and }n=5,7,11,\\
   & &  \mbox{ or }(q,n)=(3,10),\\
   \mbox{If }\eps=-, & \mbox{then } r \geq 6n+1, & \mbox{unless }q=2 \mbox{ and }n=5,9,10,\\
   & &  \mbox{ or }(q,n)=(3,9).
  \end{array}  
\end{equation}
Indeed, in the case $\eps=+$, since $n \geq 5$, 
by \cite[Theorem 3.2.2]{Tr} we have $r \geq (q^{2n-2}-1)_{\ell_1} \geq 6(n-1)+1$, unless $q=2$ and 
$2n-2 \in \{8,10,12,18,20\}$, 
or $(q,2n-2) = (3,18)$. 
Also, when $(q,n) = (2,6)$, $(2,10)$, by our choice we have $r=33$, $513$. 
Similarly, in the case $\eps=-$, 
by \cite[Theorem 3.2.2]{Tr} we have $r=(q^{2n}-1)_{\ell_1} \geq 6n+1$, unless $q=2$ and 
$2n \in \{10,12,18,20\}$, or $(q,2n) = (3,18)$. Furthermore, $r=65$ when $(q,n) = (2,6)$ by our construction.

Since  $x^Gy^G  \supseteq G \setminus \ZB(G)$, we can 
choose $x$ and $y$ such that $xy$ is a $p$-element that acts on $V$ with two Jordan blocks, of size $2n-3$ and $3$ if $2 \nmid q$, and of size 
$2n-2$ and $2$ if $2|q$.

\smallskip
It remains to show that $S=\langle u,v \rangle$. 
Assume the contrary: $H:=\langle x,y\rangle$ is contained in a maximal subgroup $M$ of $\SL_{2n}(q)$. By its choice, $x$ has 
the $\ppd(2n,q;e)$-property, with $e=2n-2$ if $\eps=+$ and $e=2n$ if $\eps=-$. Hence we can apply \cite[Main Theorem]{GPPS} to see that $M$ is in one 
of the Examples 2.1--2.9 described therein. 

\smallskip
Suppose $M$ is in (2.1). Since $\ell_1=\ppd(p,ef)$, we must have $q_0=q$. Certainly, $M$ cannot contain $\SL_{2n}(q)$, $\Sp_{2n}(q)$, or 
$\Omega^{\pm}_{2n}(q)$. The remaining possibility is that
$\SU_{2n}(\sqrt{q}) \lhd M \leq \mathrm{CU}_{2n}(\sqrt{q})$ if $2|f$. But then $\ppd(p,(2n-1)f)$ divides $|M|$ but not $|G|$, a contradiction.

\smallskip
Suppose $M$ is in (2.2), i.e. $M$ is reducible on $V$, say with a nonzero submodule $A \neq V$. 
If $\eps=-$ then $x$ is irreducible on $V$. Hence $\eps=+$, in which case 
the $\langle x \rangle$-module $V$ is an orthogonal sum of submodules: a non-degenerate subspace $U = \F_q^2$ of type $-$ (with respect
to $\QBS$) and its $(2n-2)$-dimensional orthogonal complement $U^\perp$ which is irreducible over $\langle x \rangle$. Replacing $A$ by $A^\perp$ if necessary, we may assume that $A^\perp \supseteq U^\perp$,
and hence $A \subseteq U$. If $2 \nmid q$, then, as $U$ is of type $-$, $A$ contains no (nonzero) singular vector, and so 
$A \cap A^\perp = 0$ and $V = A \oplus A^\perp$. In this case, $xy$ has a Jordan block of size $\dim(A) \leq 2$, contrary to the choice
of $x$ and $y$.  Hence $2|q$, in which case $x$ acts irreducibly on $U$ and so $A=U$. But then $y$ fixes $U = \F_q^2$ of type $-$, which is impossible
by the choice of $y$: if $2 \nmid n$ then $V|_{\langle y \rangle}$ is the sum of two irreducible submodules of dimension $n$, and if $2|n$ then 
 $V|_{\langle y \rangle}$ is the sum of two irreducible submodules of dimension $n-1$ and a $2$-dimensional submodule which is of type $+$.

\smallskip
Suppose $M$ is in (2.3). Then $\ell_1 = e+1 \leq 2n$, whence $\eps=+$, $\ell_1=2n-1$, $f=1$ and $M \leq \GL_1(q) \wr \SSS_{2n}$. Now if 
$2 \nmid n$, then $\ell_2 = \ppd(p,nf) \geq 2n+1$ and $\ell_2 \nmid (q-1)$, and so
$\ell_2 \nmid |M|$,  a contradiction. Suppose $2|n$, whence $n \geq 6$.
Then $\ell_2 = \ppd(p,n-1)$ is congruent to
$1$ module $n-1$ and different from $n$ and $3n-2$ (by parity), and from $2n-1=\ell_1$. Thus $\ell_2 \geq 4n-3$ and coprime to $q-1$,
so does not divide $|M$, again a contradiction.
 
\smallskip
Suppose $M$ is in (2.4), so that $M \leq \GL_{2n/b}(q^b) \cdot b$ for some $1 < b|2n$. Here, the first possibility is that $\ell_1=2n=b=e+1$, which is absurd.
 Hence we are in the second possibility: $b>1$ divides $\gcd(e,2n)$. Note that
\begin{equation}\label{xy2}
  xy \in \GL_{2n/b}(q^b).
\end{equation}  
(Indeed, this is certainly the case if $|x|$ and $|y|$ are both coprime to $b$,
as this ensures that $H \leq \GL_{2n/b}(q^b)$. Now, if $\eps=+$, then $e=2n-2$ and hence $b=2$, whereas by our choice $|x|$ and $|y|$ are odd.
If $\eps=-$, then $\ell_1 > 2n$ and $\ell_2 \geq 2n-1$, whereas $b=2n$ or $b \leq n$.)
Now \eqref{xy2} implies that the $p$-element $xy$ acts on $V=\F_q^{2n}$ with every Jordan block repeating $b$ times. But this contradicts the choice of 
$xy$.

\smallskip
Suppose $M$ is in (2.5). Then $2n=2^{m+1}$ with $m \geq 3$. Furthermore, $p>2$, and either $\ell_1=2n+1$ is a Fermat prime and $\eps=-$, or 
$\ell_1 =2n-1$ is a Mersenne prime and $\eps=+$. In the first case, $m \geq 3$ is odd, so $2n-1=(2^{(m+1)/2}-1)(2^{(m+1)/2}+1)$.
The structure of $M$ then shows that any odd prime divisor of $|M|$ is either $\ell_1$, a divisor of $q-1$, or 
at most $2^m+1=n+1$. On the other hand, $|y|=\ell_2 = \ppd(p,2(n-1)f) \geq 2n-1$, a contradiction. In the second case, $f=1$ and $m$ is even, so 
$3|(2n+1)$.
The structure of $M$ then shows that any odd prime divisor of $|M|$ is either $\ell_1$, a divisor of $q-1$, or 
at most $2^m+1=n+1$. On the other hand, as $n > 4$,
$\ell_2 = \ppd(p,n-1) \geq 2n-1$, again a contradiction.

\medskip
All the remaining candidates for $M$ are almost quasisimple; in particular, $L:=M^{(\infty)}$ is quasisimple acting absolutely irreducibly on $V$.
Suppose $M$ is in (2.6). Since $\ell_1 \geq 11$ by \eqref{r10}, we have that $L = \AAA_m$ acting on $V$ via its deleted permutation module
and $M \leq (C_{q-1} \times L) \cdot 2$.
Here, $2n \in \{m-2,m-1\}$ and hence all prime divisors of $M$ that are coprime to $q-1$ are at most $m$. In particular, 
$\ell_1 \leq m \leq 2n+2$. Using \eqref{r20}, we deduce that $\ell_1 = 2n \pm 1$ and $f=1$. Note that the only element order divisible 
by $\ell_1$ in $\SSS_m$ is $\ell_1$. If $\eps=+$, then by \eqref{r21} we either have $\ell_2 > 2n+2$ (and coprime to 
$q-1$) and hence $y$ cannot be contained in $M$, or $r > 2n+2 \geq m$. In the latter case $x^{\ell_1(q-1)}=1$, but $|x|=r$ is coprime to $q-1$, whence
$x^{\ell_1}=1$, which is impossible. If $\eps=-$, then by \eqref{r21} we have $\max(r,\ell_2) > 2n+2$ with $\gcd(q-1,r\ell_2)=1$, and so
we reach a contradiction as in the previous case.

\smallskip
Suppose $M$ is in (2.7). Since $2|e \geq 2n-2 \geq 8$, from \cite[Table  5]{GPPS} we see that $(e,2n,\ell_1) = (10,10,11)$, $(10,12,11)$, $(12,12,13)$, $(18,18,19)$, $(18,20,19)$, $(22,22,23)$, $(22,24,23)$, or $(28,28,29)$. In the three cases with $e=2n-2$: $(10,12,11)$, $(18,20,19)$, and $(22,24,23)$, we have
$\eps=+$, $2|n$, $\ell_2 \equiv 1 \pmod{n-1}$ and $\ell_2 \neq \ell_1=2n-1$, which together imply $\ell_2 \geq 31$ and so $\ell_2$ does not divide $|M|$, a
contradiction. In the following three cases with $e=2n$ (so $\eps=-)$: $(10,10,11)$, $(22,22,23)$, and $(28,28,29)$, we have
$\ell_2 \geq 17$, $41$, $53$, respectively, and so $\ell_2$ does not divide $|M|$, again a
contradiction. In the case of $(12,12,13)$, i.e. $L = 6 \cdot \mathrm{Suz}$, we have $p > 2$, so by \eqref{r20} the power $r$ of 
$\ell_1=13$ satisfies $r \geq 37$, i.e. $r \geq \ell_1^2$.  
In the case of $(18,18,19)$, i.e. $L = 3 \cdot \mathrm{J}_3$, we have $p > 3$, so by \eqref{r20} the power $r$ of 
$\ell_1=19$ satisfies $r \geq 55$, i.e. $r \geq \ell_1^2$. In both of these two cases, $r=|x|$ is coprime to $q-1$, and hence $M$ cannot have any element of such order $r \geq \ell_1^2$, a contradiction.

\smallskip
Suppose we are in (2.8). Since $n \geq 5$, we have no example from \cite[Table 6]{GPPS}. 

\smallskip
Finally, suppose $M$ is in (2.9). 
Since $n \geq 5$, for the groups in \cite[Table 7]{GPPS} we see that $(2n,e,\ell_1,L)$ is one of
$(12,12,13,2 \cdot G_2(4))$, $(14,12,13,\tw2 B_2(8) \mbox{ or } G_2(3))$, and $(18,16,17,\Sp_4(4))$. In the first two cases, 
$\ell_2 = \ppd(p,5f) \geq 11$ cannot divide $|M|$. In the third case, 
$\ell_2 = \ppd(p,8f) \neq 17$, so $\ell_2 \geq 41$ cannot divide $|M|$.

Next we consider the possibilities for $L$ listed in \cite[Table 8]{GPPS}. We will again use the upper bounds for $\meo(L)$ listed in \cite[Table 3]{GMPS}. The construction of our elements $x$ and $y$ ensures that their order is the same order in 
$\PGL(V)$, and hence in $\Aut(L)$ as $L$ acts absolutely irreducibly on $V$.
The first case is $L/\ZB(L)=\PSL_m(t)$ with $m \geq 3$ a prime and 
$$2n+1 \geq e+1=\ell_1= (t^m-1)/(t-1) \geq 2n,$$ 
so in fact $\ell_1=2n+1$, $e=2n$, and $\eps=-$. Here, $\meo(L)=\ell_1\leq 2n+1$,  so by \eqref{r20} we must have $r=\ell_1$, and hence we are in one of the exceptions of \eqref{r21}. Thus 
$2n=(t^m-t)/(t-1) \in \{10,18\}$ (as $\ell_1 =2n+1 \neq 21$), which is impossible.

The second case is $L/\ZB(L)=\PSU_m(t)$ with $m \geq 3$ a prime and 
$2n+1 \geq e+1=\ell_1= (t^m+1)/(t+1) \geq 2n$, so in fact $\ell_1=2n+1$, $e=2n$, and $\eps=-$. Here, $\meo(L) \leq 2\ell_1+2 < 3\ell_1$, so by \eqref{r20} we must have $r=\ell_1$, and hence we are in one of the exceptions of \eqref{r21}. Thus 
$2n=(t^m-t)/(t+1) \in \{10,18\}$  (as $\ell_1 =2n+1 \neq 21$), which is possible only when $2n=10$ and $L = \SU_5(2)$. But in this case 
$|y|=\ell_2=\ppd(p,8f) \geq 17$, and so $\ell_2 \nmid |M|$, again a contradiction. 

The third family is $L/\ZB(L)=\PSp_{2m}(t)$, where either $\ell_1=(t^m+1)/2$ and $2 \nmid t$, or $\ell_1=(3^m-1)/2$, $t=3$, and $m$ is a prime. 
It is easy to see that $t=t_0^{2^a}$ for a prime $t_0$ in the first case. As $H=\OB^2(H)$, in either case we have $H \leq \OB^2(M) \leq \ZB(M)L$
(as $\Out(L)$ is a $2$-group). Now we can check that the element $x$ of order $\ell_1$ generates a maximal torus in $L/\ZB(L)$, and hence $r=\ell_1=e+1=2n+1$, which implies $n \in \{5,7,9,10,11\}$ by \eqref{r21}. As $2n=e=\ell_1-1$, this cannot happen for the indicated $\ell_1=(t^m \pm 1)/2$.

In the remaining families we have $L/\ZB(L)= \PSL_2(t)$ with $t \geq 7$. Since $\max(\ell_1,\ell_2) \geq 11$ by \eqref{r20}, in fact we have $t \geq 11$, and 
$\meo(L) = t+1$ by \cite{GMPS}. Here, $(t-1)/2 \leq 2n$, so $r \leq t+1 \leq 4n+2 < 6n-5$. Hence we must be in one of the exceptions listed 
in \eqref{r21}, and thus $n \in \{5,7,9,10,11\}$. In the case $t=2^c$ we have $2n \in \{t,t \pm 1\}$, and none fits in the indicated range. So $2 \nmid t$.
We will now take into account the value of $q$ listed in these exceptions. When $\eps=+$, we have $q=2$ and $n = 5,7,11$, or $(q,n) = (3,10)$, in which cases
we can take $\ell_2 = 31$, $127$, $89$, or $257$, respectively, which all exceed $\meo(L) \leq 4n+2$. When $\eps=-$, we have $q=2$ and $n = 5,9,10$, or $(q,n) = (3,9)$, in which cases we can take $\ell_2 = 17$, $257$, $19$, or $193$, respectively. Since $\meo(L) \leq 4n+2$, we are left with the  
two cases $(q,n) = (2,5)$, $(2,10)$. If $(q,n) = (2,5)$, then $t = 9$, $11$, or $19$, and so $\ell_2=17$ does not divide $|\Aut(L)|$. 
If $(q,n) = (2,10)$, then $t = 19$, or $41$. Since $\ell_1=41$ divides $|\Aut(L)|$, we get $t=41$. But in this case $\ell_2=19$ again does not 
divide $|\Aut(L)|$, the final contradiction.

This completes the proof of Theorem \ref{main1} and Theorem \ref{main2} unless $S$ is not a sporadic simple group.

\section{Proof of Theorems \ref{main1}  and \ref{main2}: The sporadic groups}

Here we record some data describing our verification of Theorems \ref{main1} and \ref{main2} for the sporadic simple groups. Typically, we will begin each entry with the name of the group $S$, followed by its order $N$, the isomorphism type of its Schur multiplier $H_2(S,\Z)$, and the \cite{GAP} command for constructing a certain perfect central extension of the group. This central extension will be denoted $G$, the simple group being $S = G/Z(G)$. If the center of the central extension is not specified, then the extension will be the Schur cover. The end of each entry will be a statement which implies that Theorems \ref{main1} and \ref{main2} hold for the simple group $S$.  Unless otherwise mentioned, our methods employ a random search via the product replacement algorithm using \cite{GAP}.

The default generators of this central extension will refer to the generating pair which can be obtained via the \cite{GAP} command \texttt{GeneratorsOfGroup(<group>);}. For a triple of integers $r,s,t\in\Z_{\ge 1}$, an $(r,s,t)$-pair of a group $G$ is a pair $x,y\in G$ satisfying $|x| = r, |y| = s, |xy| = t$.

\subsection{Mathieu groups}
\begin{itemize}
\item $M_{11}$, $N = 7920$, $H_2(S,\Z) = 0$, \texttt{AtlasGroup("M11");}

$G$ admits a $(3,8,11)$-generating pair Nielsen equivalent to the default generators.
\item $M_{12}$, $N = 95040$, $H_2(S,\Z) = C_2$, \texttt{AtlasGroup("2.M12");}

$G$ admits a $(3,5,11)$-generating pair Nielsen equivalent to the default generators.

\item $M_{22}$, $N = 443520$, $H_2(S,\Z) = C_{12}$, \texttt{AtlasGroup("4.M22");}, $Z(G) \cong C_4$.

$G$ admits a $(3,5,11)$-generating pair Nielsen equivalent to the default generators.

\item $M_{23}$, $N = 10200960$, $H_2(S,\Z) = 0$, \texttt{AtlasGroup("M23");}

$G$ admits a $(7,8,15)$-generating pair Nielsen equivalent to the default generators.

\item $M_{24}$, $N = 244823040$, $H_2(S,\Z) = 0$, \texttt{AtlasGroup("M24");}

The default generators of $G$ is a $(2,3,23)$-pair.

\end{itemize}

\subsection{Janko groups}
\begin{itemize}
\item $J_1$, $N = 175560$, $H_2(S,\Z) = 0$, \texttt{AtlasGroup("J1");}

The default generators of $G$ is a $(2,3,7)$-pair. Thus $J_1$ is a Hurwitz group, see \cite{Con2}.
\item $J_2$, $N = 604800$, $H_2(S,\Z) = C_2$, \texttt{AtlasGroup("2.J2");}

$J_2$ is a Hurwitz group, see \cite{Con2}; in fact the default generators of \texttt{AtlasGroup("J2");} is a $(2,3,7)$-pair. The Schur cover $G$ admits a $(7,7,3)$-generating pair which is Nielsen equivalent to its default generators.

\item $J_3$, $N = 50232960$, $H_2(S,\Z) = C_3$, \texttt{AtlasGroup("3.J3");}

$G$ admits a $(8,17,19)$-generating pair Nielsen equivalent to the default generators.

\item $J_4$, $N = 86775571046077562880$, $H_2(S,\Z) = 0$, \texttt{AtlasGroup("J4");}

$J_4$ is a Hurwitz group \cite{Con2}, and thus admits a $(2,3,7)$-generating pair.
\end{itemize}

\subsection{Conway groups}

\begin{itemize}
\item $\text{Co}_1$, $N = 4157776806543360000$, $H_2(S,\Z) = C_2$, \texttt{AtlasGroup("2.Co1");}

$G$ admits both a $(15,15,23)$-generating pair as well as a $(13,20,21)$-generating pair.

\item $\text{Co}_2$, $N = 42305421312000$, $H_2(S,\Z) = 0$, \texttt{AtlasGroup("Co2");}

$G$ admits a $(5,9,28)$-generating pair Nielsen equivalent to the default generators.

\item $\text{Co}_3$, $N = 495766656000$, $H_2(S,\Z) = 0$, \texttt{AtlasGroup("Co3");}

$\text{Co}_3$ is a Hurwitz group \cite{Con2}, and thus admits a $(2,3,7)$-generating pair.
\end{itemize}

\subsection{Fischer groups}
\begin{itemize}
\item $\text{Fi}_{22}$, $N = 64561751654400$, $H_2(S,\Z) = C_6$, \texttt{AtlasGroup("3.Fi22");}, $Z(G)\cong C_3$.

$G$ admits a $(7,11,13)$-generating pair Nielsen equivalent to the default generators.
\item $\text{Fi}_{23}$, $N = 4089470473293004800$, $H_2(S,\Z) = 0$, \texttt{AtlasGroup("Fi23");}

$G$ admits a $(11,12,35)$-generating pair Nielsen equivalent to the default generators.
\item $\text{Fi}_{24}'$, $N = 1255205709190661721292800$, $H_2(S,\Z) = C_3$, \texttt{AtlasGroup("3.Fi24'");}

$G$ admits a $(13,17,20)$-generating pair Nielsen equivalent to the default generators.
\end{itemize}

\subsection{Other sporadic groups}

\begin{itemize}
\item Higman-Sims group $\text{HS}$, $N = 44352000$, $H_2(S,\Z) \cong C_2$, \texttt{AtlasGroup("2.HS");}

$G$ admits a $(5,7,11)$-generating pair Nielsen equivalent to the default generators.

\item McLaughlin group $\text{McL}$, $N := 898128000$, $H_2(S,\Z) \cong C_3$, \texttt{AtlasGroup("3.McL");}

$G$ admits a $(5,7,11)$-generating pair Nielsen equivalent to the default generators.

\item Held group \text{He}, $N = 4030387200$, $H_2(S,\Z) = 0$, \texttt{AtlasGroup("He");}

This is a Hurwitz group \cite{Con2}, and thus admits a $(2,3,7)$-generating pair. 

\item Rudvalis group \text{Ru}, $N = 145926144000$, $H_2(S,\Z) \cong C_2$, \texttt{AtlasGroup("2.Ru");}

$G$ admits a $(5,13,29)$-generating pair Nielsen equivalent to the default generators.

\item Suzuki sporadic group $\text{Suz}$, $N = 448345497600$, $H_2(S,\Z) \cong C_6$, \texttt{AtlasGroup("2.Suz");}, $Z(G)\cong C_2$.

$G$ admits a $(5,13,21)$-generating pair Nielsen equivalent to the default generators.

\item O'Nan group \text{O'N}, $N = 460815505920$, $H_2(S,\Z)\cong C_3$, \texttt{AtlasGroup("3.ON");}

$G$ admits a $(19,20,31)$-generating pair Nielsen equivalent to the default generators.

\item Harada-Norton group \text{HN}, $N = 273030912000000$, $H_2(S,\Z) = 0$, \texttt{AtlasGroup("HN");}

\text{HN} is a Hurwitz group \cite{Con2}, and hence admits a $(2,3,7)$-generating pair.

\item Lyons group \text{Ly}, $N = 51765179004000000$, $H_2(S,\Z) = 0$, \texttt{AtlasGroup("Ly");}

\text{Ly} is a Hurwitz group \cite{Con2}, and hence admits a $(2,3,7)$-generating pair.

\item Thompson group \text{Th}, $N = 90745943887872000$, $H_2(S,\Z) = 0$, \texttt{AtlasGroup("Th");}

\text{Th} is a Hurwitz group \cite{Con2}, and hence admits a $(2,3,7)$-generating pair.

\item Baby Monster group \text{B}, $N = 4154781481226426191177580544000000$, $H_2(S,\Z)\cong C_2$.

The Baby Monster can be constructed in \cite{GAP} as \texttt{AtlasGroup("B");}, but its Schur double cover cannot. Nonetheless, one can access the character table of the double cover. The following \cite{GAP} transcript show that the Schur cover of the Baby Monster admits a $(3,5,47)$-pair:

\begin{tabular}{l}
\texttt{gap> tbl := CharacterTable("2.B");} \\
\texttt{CharacterTable( "2.B" )} \\
\texttt{gap> irr := Irr(tbl);;} \\
\texttt{gap> OrdersClassRepresentatives(tbl)\{[8,23,228]\};} \\
\texttt{[ 3, 5, 47 ]} \\
\texttt{gap> Sum(List(irr,ch -> ch[8]*ch[23]*ch[228]/ch[1]));} \\
\texttt{34626119/41013248}
\end{tabular}

Here, the fourth and fifth lines shows that 8th, 23rd, and 228th conjugacy classes (in the table \texttt{tbl}) consist of elements of orders 3,5,47 respectively. Lines 6-7 then show that if $x,y,z$ are representatives of such conjugacy classes, then
$$\sum_{\chi \text{ irr}}\frac{\chi(x)\chi(y)\chi(z)}{\chi(1)} > 0$$
This implies that the Schur cover of the Baby Monster admits a $(3,5,47)$-pair \cite[Theorem 7.2.1]{Serre}. Finally, from the list of the maximal subgroups of the Baby Monster \cite{Atlas}, we find that up to conjugation, the only maximal subgroup which contains an element of order 47 has order $23\cdot 47$. It follows that any (3,5,47)-pair is generating.

\item Monster group \text{M}, $N = 808017424794512875886459904961710757005754368\cdot 10^9$.

The monster \text{M} has a trivial Schur multiplier, so it suffices to note that it is a Hurwitz group \cite{Con2}, and hence admits a $(2,3,7)$-generating pair.

\end{itemize}

\end{document}